\documentclass[11pt,a4paper]{amsart}
\usepackage[a4paper,centering,left=2.2cm,right=2.2cm]{geometry}
\usepackage{amsfonts,amscd,amssymb,amsmath,amsthm,mathrsfs,multirow,xcolor,lscape,longtable,pbox,lipsum}
\usepackage[colorlinks,linkcolor=blue,anchorcolor=blue,citecolor=blue,backref=page]{hyperref}
\usepackage{hyperref}
\usepackage[thinlines]{easytable}
\usepackage{microtype}
\newtheorem{thm}{Theorem}

\newtheorem{lema}[thm]{Lemma}
\newtheorem{prop}[thm]{Proposition}

 \usepackage[all,cmtip]{xy}\usepackage{tikz}\usetikzlibrary{arrows} 
\usetikzlibrary{matrix}\usetikzlibrary{calc}

\newcommand\scalemath[2]{\scalebox{#1}{\mbox{\ensuremath{\displaystyle #2}}}}

\newcommand{\Z}{\mathbb{Z}}
\newcommand{\Q}{\mathbb{Q}}

\newcommand{\G}{\Gamma}
\newcommand{\N}{\mathbb{N}}
\newcommand{\T}{\mathrm{T}}
\newcommand{\R}{\mathbb{R}}

\newcommand{\m}{\mathcal{M}}

\newcommand{\W}{\mathcal{W}}

\newcommand{\rG}{\mathrm{G}}
\newcommand{\rK}{\mathrm{K}}
\newcommand{\rP}{\mathrm{P}}
\newcommand{\rU}{\mathrm{U}}
\newcommand{\rM}{\mathrm{M}}
\newcommand{\rS}{\mathrm{S}}

\newcommand{\tmt}[4]{\left({#1\atop #3}{#2\atop #4}\right)}

\newcommand{\eqr}[1]{\mbox{(\ref{eq:#1})}}
\newcommand{\ie}{i.e.\ }
\newcommand{\mr}[1]{\mathrm{#1}}
\newcommand{\C}{\mathbb{C}}

\newcommand{\tm}{\widetilde{\m}}

\newcommand{\Sp}{\mathrm{Sp}_4}

\title[Euler Characteristic and  Cohomology of $\Sp(\Z)$ with nontrivial coefficients] {Euler Characteristic and  Cohomology of $\Sp(\Z)$ with nontrivial coefficients} 
\author{Jitendra Bajpai, Ivan Horozov, Matias Moya Giusti} 
\address{Mathematisches Institut, Georg-August Universit\"at G\"ottingen, D-37073 Germany.}
\email{jitendra@math.uni-goettingen.de}
\curraddr{Institut f\"ur Geometrie, Technische Universit\"at Dresden, Germany}
\email{jitendra.bajpai@tu-dresden.de}
\address{Graduate Center, City University of New York, 365 5th Ave, New York, NY 10016, USA.}
\curraddr{BCC, City University of New York, 2155 University Ave, Bronx, NY 10453, USA.}\email{ivan.horozov@bcc.cuny.edu}
\address{LAMA - Universit\'e Paris-Est Marne-la-Vall\'ee, 77454 - Marne-la-Valle, France.} 
\email{matias-victor.moyagiusti@u-pem.fr}

\begin{document}
\date{\today}

\begin{abstract}

In this article, the cohomology of the arithmetic group $\Sp(\Z)$ with coefficients in any highest weight irreducible representation $\m_\lambda$ has been studied.  Euler characteristic with coefficients in $\m_\lambda$ has been carried out in detail. Combining the results obtained on Euler characteristic and the work of Harder on Eisenstein cohomology~\cite{Harder2012}, the description of the cuspidal cohomology has been achieved. 
At the end, we employ our study to compute the dimensions for the cohomology spaces $H^{\bullet}(\Sp(\Z), \m_\lambda)$.

\end{abstract}

\maketitle

\tableofcontents

  
\section{Introduction}\label{intro} 
The cohomology theory of arithmetic groups has played an important role in understanding automorphic forms and the geometry associated to our underlying arithmetic group. An important problem in this theory is to obtain the description of the full cohomology with respect to arbitrary coefficient systems in terms of automorphic forms. This question is quite difficult in general. In this article, we set our focus on the arithmetic group $\Sp(\Z)$.

Let $\rG$ be a semisimple algebraic group defined over $\Q$, $\rK_\infty \subset \rG(\R)$ be a maximal compact subgroup and $\rS = \rG(\R)/\rK_\infty$ be the corresponding symmetric space. If $\Gamma \subset \rG(\Q)$ is an arithmetic subgroup then every representation $(\rho, \m)$ of $\rG_\C$ defines in a natural way a sheaf $\tm$ on the locally symmetric space $\rS_\Gamma = \Gamma \backslash \rS$. One has the isomorphism $H^\bullet(\Gamma, \m) \cong H^\bullet(\rS_\Gamma, \tm),$ (for details see Chapter 7 of~\cite{BoWa}). On the other hand, let $\overline{\rS}_\Gamma$ denote the Borel-Serre compactification of $\rS_\Gamma$, then the inclusion $i:\rS_\Gamma \hookrightarrow\overline{\rS}_\Gamma$, which is an homotopy equivalence, determines a sheaf $i_\ast\tm$  on $\overline{\rS}_\Gamma$ and induces a canonical isomorphism in cohomology
\begin{equation} \label{eq:Isomorphism BS}
	H^\bullet(\rS_\Gamma, \tm) \cong H^\bullet(\overline{\rS}_\Gamma, i_\ast \tm).
\end{equation}
From now on we will simply denote $i_\ast\tm$ by $\tm$. For details on Borel-Serre compactification we refer the interested reader to ~\cite{BoSe73, Harder2018}.

The choice of a maximal $\Q$-split torus $\T$ of $\rG$ and a system of positive roots $\Phi^+$ in $\Phi(\rG, \T)$ determines a set of representatives for the conjugacy classes of $\Q$-parabolic subgroups, denoted by $\mathcal{P}_\Q(\rG)$, namely the standard $\Q$-parabolic subgroups. The boundary of the Borel-Serre compactification $\partial \rS_\Gamma = \overline{\rS}_\Gamma - \rS_\Gamma$ is a finite union of the boundary components $\partial_{\rP}$ for $\rP \in \mathcal{P}_\Q(\rG)$, i.e.
\begin{equation} \label{SpesSeqIntro}
\partial \rS_\Gamma = \bigcup_{\rP \in \mathcal{P}_\Q(\rG)} \partial_{\rP}, \nonumber
\end{equation}
and this covering determines a spectral sequence in cohomology abutting to the cohomology of the boundary
\begin{equation}
E^{p, q}_{1} = \bigoplus_{prk(\rP)=p+1} H^q(\partial_{\rP}, \tm) \Rightarrow H^{p+q}(\partial \rS_\Gamma, \tm)
\end{equation}
where $prk(\rP)$ denotes the parabolic rank of $\rP$ (the dimension of the maximal $\Q$-split torus in the center of the Levi quotient $\rM$ of $\rP$). When the $\Q$-rank of $\rG$ is $2$, the aforementioned spectral sequence is just a long exact sequence in cohomology. 

In this article we set our focus on the arithmetic group $\Gamma = \Sp(\Z)$. We use the spectral sequence ~\eqref{SpesSeqIntro} to determine the cohomology space of the boundary of the Borel-Serre compactification of the locally symmetric space $\rS_\Gamma$ associated to the arithmetic subgroup $\Sp(\Z) \subset \Sp(\R)$ and describe the Eisenstein cohomology by following ~\cite{Harder2012}. We continue to determine the Euler characteristic with respect to any finite dimensional irreducible representation $\m_{\lambda}$ of $\Sp$ (for this we use ~\cite{Horozov2005}). 

As an immediate application of the results obtained on Euler characteristic we are able to determine the dimension of the cuspidal cohomology of the symmetric space $\rS_\Gamma$ associated to $\Sp(\Z)$. Another application is the description of the dimension of the cohomology spaces of $\Sp(\Z)$ with respect to every finite dimensional highest weight representation of $\Sp$.

We can summarize the main ideas and results of this paper in a few lines. We use a formula in \cite{Horozov2005} to determine the Euler characteristic of the group $\Sp(\Z)$ with respect to every $\m_\lambda$ finite dimensional irreducible representation of $\Sp$ with highest weight $\lambda$. As we have already mentioned, one has an isomorphism of the cohomology spaces $H^q(\Sp(\Z), \m_\lambda) \cong H^q(\rS_\Gamma, \tm_\lambda)$. On the other hand, one has a decomposition of the cohomology of $\rS_\Gamma$ as the direct sum of the cuspidal and the Eisenstein cohomology
\[
H^q(\rS_\Gamma, \tm_\lambda \otimes \C) = H_{cusp}^q(\rS_\Gamma, \tm_\lambda \otimes \C) \oplus H_{Eis}^q(\rS_\Gamma, \tm_\lambda \otimes \C).
\]
One could consider the contributions $\chi_{Eis}(\lambda), \chi_{cusp}(\lambda)$ to the Euler characteristic comming from the Eisenstein and cuspidal part, respectively. Therefore the homological Euler  characteristic $\chi_h(\lambda)$ of $\Sp(\Z)$ with respect to $\m_\lambda$, is the sum of $\chi_{Eis}(\lambda)$ and $\chi_{cusp}(\lambda)$. The Eisenstein cohomology has been determined by Harder in \cite{Harder2012}, so one can give a formula for $\chi_{Eis}(\lambda)$. On the other hand, by using \cite{MokTil2002}, one knows that the cuspidal cohomology is always concentrated in degree $3$. Therefore one obtains the identity
\[
	dim(H_{cusp}^3(\rS_\Gamma, \tm_\lambda \otimes \C)) = \chi_h(\lambda) - \chi_{Eis}(\lambda)
\]
determining a formula to calculate the dimension of the cuspidal cohomology. Let $H_{!}^q(\rS_\Gamma, \tm_\lambda)$ denote the inner cohomology (that is the kernel of the natural restriction morphism $r^q:H^q(\rS_\Gamma, \tm_\lambda) \rightarrow H^q(\partial \rS_\Gamma, \tm_\lambda)$). In the case of $\Sp$ one has, by \cite{MokTil2002}, that  $H_{!}^q(\rS_\Gamma, \tm_\lambda) \otimes \C \cong H_{cusp}^q(\rS_\Gamma, \tm_\lambda \otimes \C)$, therefore this also determines a formula for the dimension of the inner cohomology. By similar arguments one can also calculate the dimension of the group cohomology of $\Sp(\Z)$.

The main results of the paper are given in:
\begin{itemize}
\item Theorem ~\ref{Hom Symm}, where we describe a formula for the Euler characteristics with respect to the symmetric power representations. 
\item Theorem ~\ref{Homological formula}, describing a formula for the Euler characteristics for general highest weights.
\item Theorem ~\ref{Cuspidal Dim}, where a formula for the dimension of the cuspidal cohomology is determined.
\item Theorem ~\ref{Dim}, that gives a formula for the dimensions of the group cohomology spaces $H^\bullet(\Sp(\Z), \m_\lambda)$.
\end{itemize}

We end this section by giving a quick overview on the organization of this article. 

In Section~$\ref{preli}$ we introduce the notation and present some brief introduction to the results that we use in the paper.

In Section~\ref{nerve} we give a full description of the boundary cohomology of the Borel-Serre compactification of the locally symmetric space associated to $\Sp(\Z)$. We took this opportunity to discuss about the boundary cohomology in detail and in the form we wanted in this article. In this section we are also using the methods described in \cite{BHHM2018}. We feel that the approach we took is more naive and easy     to follow. However, the details could easily be traced down following the work of Harder in~\cite{Harder2012}. 

Section~\ref{Eisenstein} gives the basics about the Eisenstein cohomology and summarize the results from~\cite{Harder2012}. This plays a crucial role in determining the dimension of the space of cuspidal cohomology and the dimension of the cohomology of $\Sp(\Z)$ discussed in the last two sections, \ie Sections ~\ref{Dim cuspidal} and ~\ref{Dim Coho}.

Section~\ref{Euler char} supplies one of the most important tools to achieve the goal. We make use of Yang's work~\cite{Yang}, on classification of torsion elements of $\Sp(\Z)$ to compute their centralizers which eventually gives us the orbifold Euler characteristics. 

In Section ~\ref{Traces} we compute the trace of each representative of the conjugacy classes of torsion elements in $\Sp(\Z)$ with respect to any finite dimensional irreducible representation of $\Sp$. We use these results to give the full description of the homological Euler characteristic of $\Sp(\Z)$ in Section~\ref{homeuler}. Finally in the last two sections we introduce some important applications of the results obtained. First, in Section \ref{Dim cuspidal} we give the description of the dimensions of the cuspidal cohomology of $\rS_\Gamma$. Finally, in the last section we determine the dimension of the cohomology spaces of $\Sp(\Z)$ with respect to every highest weight finite dimensional irreducible representation.

In this article we use the notation $H^\bullet$ to denote the direct sum of the cohomology spaces in every degree. So, for example, with this notation
\[
	H^\bullet(\Sp(\Z), \widetilde{\mathcal{M}}_\lambda) = \bigoplus_{q \geq 0} H^q(\Sp(\Z), \widetilde{\mathcal{M}}_\lambda).
\]
Among the many related works in the literature (see \cite{GrossPollack}, \cite{Faltings90}, \cite{Tsushima83}), we would like to mention the papers \cite{Dan2015} and \cite{Taibi2017} where the cuspidal cohomology of $\Sp(\Z)$ is also calculated. The main differences with these papers are the methods used and the explicit and simple dimension formulas we present among the main theorems. In fact, the importance of this article lies in providing significant and practical dimension formulas and the simplicity of the approach introduced. 


\section{Preliminaries}\label{preli}
This section quickly review the basic properties of $\Sp$ and familiarize the reader with the notations to be used throughout the article. We discuss the corresponding locally symmetric space, Weyl group, the associated spectral sequence and Kostant representatives of the standard parabolic subgroups.

\subsection{Structure Theory}\label{sp4}
Consider the algebraic group $\Sp$ over $\Q$ which is  defined for every $\Q$-algebra $A$ by
\[
	\Sp(A) =  \left\{ g \in \mr{GL}_4(A) \mid g^t J g = J  \right\} ,
\]
 where
\[J = \left( \begin{array}{cc}
0 & id_2 \\
-id_2 & 0 \end{array} \right) \mbox{, with }
id_2 = \left( \begin{array}{cc}
1 & 0 \\
0 & 1 \end{array} \right).
\]

Consider the arithmetic subgroup $\Sp(\Z) \subset \Sp(\Q)$ and the maximal compact subgroup $\rK_\infty \subset \Sp(\R)$ defined by 
\[
	\rK_\infty = \left\{ \left( \begin{array}{cc}
A & B \\
-B & A \end{array}  \right) \mid A + i B \in \rU_2(\mathbb{C})\right\} \subset \Sp(\R)\,.
\] 
From now on throughout the article let $\Gamma$ denote the arithmetic group $\Sp(\Z)$. Consider the symmetric space $\rS = \Sp(\R) / \rK_\infty$ and let $\m$ be a representation of $\Sp$. $\G$ acts naturally on $\m$ and defines a sheaf $\tm$ on $\rS_\Gamma = \Gamma \backslash \rS$. As mentioned in the introduction, one has an isomorphism
 \[
 H^\bullet(\Gamma, \m) \cong H^\bullet(\rS_\Gamma, \tm) . 
 \]
and as discussed earlier, if $\overline{\rS}_\Gamma$ is the Borel-Serre compactification of $\rS_\Gamma$ then following ~\eqr{Isomorphism BS}, we get a natural isomorphism between the cohomology spaces $H^\bullet(\overline{\rS}_\Gamma, \tm)$ and $ H^\bullet(\rS_\Gamma, \tm)$.

\subsection{Root System}

Consider the maximal torus $\mathrm{H}$ of $\Sp(\C)$ defined by the subgroup of diagonal matrices $\left\{ diag( h_1,  h_2, h_1^{-1}, h_2^{-1}) \mid h_1, h_2 \in \mathbb{C}^\ast\right\} \subset \Sp(\C)$. Let $\mathfrak{g}$ denote the Lie algebra $\mathfrak{sp}_{4}$ and let $\mathfrak{h} \subset \mathfrak{g}_\mathbb{C}$ be the complex Lie subalgebra associated to $\mathrm{H}$. The root system $\Phi = \Phi(\mathfrak{g}_\mathbb{C}, \mathfrak{h})$ is of type $C_2$. Let $\varepsilon_1, \varepsilon_2 \in \mathfrak{h}^\ast$ be defined by $\varepsilon_1(X) = h_1$ and $\varepsilon_2(X)=h_2$ for $X=diag(h_1,h_2,-h_1,-h_2) \in \mathfrak{h}$. Then the root system $\Phi$ is given by $\left\{\pm \varepsilon_1 \pm \varepsilon_2, \pm 2\varepsilon_1, \pm 2\varepsilon_2 \right\}$, a set of positive roots $\Phi^{+}$ is $ \left\{\varepsilon_1+\varepsilon_2, \varepsilon_1-\varepsilon_2, 2\varepsilon_1, 2\varepsilon_2\right\}$ and the system of simple roots $\Delta$ is $\left\{\alpha_1=\varepsilon_1-\varepsilon_2,\alpha_2= 2\varepsilon_2\right\}$. Finally, we denote $\delta = \frac{1}{2}\sum_{\alpha \in \Phi^+} \alpha = 2 \varepsilon_1 + \varepsilon_2$.

\subsection{Standard $\mathbb{Q}$-Parabolic Subgroups}\label{standard}

The standard $\Q$-parabolic subgroups of $\Sp$ with respect to the given $\mathbb{Q}$-root system and system of positive roots will be described in this subsection. As $\Delta$ has just two elements, we have three proper standard $\mathbb{Q}$-parabolic subgroups, one minimal and two maximal ones. The maximal $\Q$-parabolics $\rP_1, \rP_2$ are given by 
\[ 
\rP_1(A) = \left\{ \left( \begin{array}{cccc}
\ast & \ast & \ast & \ast \\
0 & \ast & \ast & \ast \\
0 & 0 & \ast & 0 \\
0 & \ast & \ast & \ast \end{array}  \right) \in \mr{GL}_4(A) \right\} \cap \Sp(A)
\]
and
\[ 
\rP_2(A) = \left\{ \left( \begin{array}{cccc}
\ast & \ast & \ast & \ast \\
\ast & \ast & \ast & \ast \\
0 & 0 & \ast & \ast \\
0 & 0 & \ast & \ast \end{array}  \right) \in \mr{GL}_4(A) \right\} \cap \Sp(A),
\]
for every $\Q$-algebra $A$.

Note that the minimal  $\Q$-parabolic $\rP_0$ is simply the group $\rP_1 \cap \rP_2$. Therefore, the corresponding Levi quotients are given by
\[ 
\rM_0 =  \mathbb{G}_m \times \mathbb{G}_m, \quad \rM_1 = \mathbb{G}_m \times \mr{SL}_2 \quad \mbox{ and } \quad \rM_2 = \mr{GL}_2.
\]

\subsection{The Irreducible Representations}\label{irreducible}
The root system $\Phi$ with the usual system of positive roots $\Phi^+$, has fundamental weights $\lambda_1, \lambda_2: \mathfrak{h} \longrightarrow \C$ given by $\lambda_1=\varepsilon_1$ and $\lambda_2=\varepsilon_1+\varepsilon_2$. Thus the irreducible finite dimensional representations of $\Sp$ are determined by their highest weights, which in this case are the linear functionals of the form $m_1 \lambda_1 + m_2 \lambda_2$ with $m_1, m_2$ non-negative integers. We fix an irreducible algebraic representation $(r_\lambda, \m_\lambda)$ of $\Sp$ with highest weight $\lambda=m_1\lambda_1+m_2\lambda_2$. Since our group $\Sp$ is $\Q$-split, the representation $\m_\lambda$ is defined over $\Q$ and we will consider $\m_\lambda$ to be the corresponding $\Q$-vector space.

\subsection{Kostant Representatives} \label{KostantRep}
It is known that the Weyl group $\mathcal{W}=\mathcal{W}(\mathfrak{g}, \mathfrak{h})$ is given by $8$ elements. They are listed in the first column of Table \ref{table1} and described in the second column as a product of simple reflections $s_1$ and $s_2$, associated to the simple roots $\alpha_1$ and $\alpha_2$ respectively. In the third column we make a note of their lengths and in the last column we describe the element $w \cdot \lambda = w(\lambda + \delta) - \delta$, where the pair $(a, b)$ denotes the element $a\varepsilon_1 + b\varepsilon_2 \in \mathfrak{h}^\ast$.  

\begin{table}[ht]
\centering
\begin{tabular}{clcl}
\hline \noalign{\smallskip}
Label & $w$  & $\ell(w)$ & $w \cdot \lambda$\\
\noalign{\smallskip}\hline\noalign{\smallskip}
$w_0$ & $1$ & $0$ & $(m_1+m_2, m_2)$\\ 
$w_1$ & $s_1$ & $1$ & $(m_2-1, m_1+m_2+1)$\\
$w_2$ & $s_2$ & $1$ & $(m_1+m_2, -m_2-2)$\\
$w_3$ & $s_1 \circ s_2$ & $2$ & $(-m_2-3, m_1+m_2+1)$\\
$w_4$ & $s_2 \circ s_1$  & $2$ & $(m_2-1, -m_1-m_2-3)$\\
$w_5$ & $s_1 \circ s_2 \circ s_1$ & $3$ & $(-m_1-m_2-4, m_2)$\\
$w_6$ & $s_2 \circ s_1 \circ s_2$  & $3$ & $(-m_2-3, -m_1-m_2-3)$\\
$w_7$ & $s_1 \circ s_2 \circ s_1 \circ s_2$  & $4$ & $(-m_1-m_2-4, -m_2-2)$\\
\noalign{\smallskip}\hline
\end{tabular}
\vspace{0.4cm}
\caption{The Weyl Group of $\Sp$}\label{table1}
\end{table}

Let $\Phi^+$ and $\Phi^-$ denote the set of positive and negative roots respectively. The Weyl group acts naturally on the set of roots. For each $i \in \left\{0, 1, 2\right\}$, let $\Delta(\mathfrak{u}_i)$ denote the set consisting of every root whose corresponding root space is contained in the Lie algebra  $\mathfrak{u}_i$ of the unipotent radical of $\rP_i$. The set of Weyl representatives $\W^{\rP_i} \subset \W$ associated to the parabolic subgroup $\rP_i$ (see \cite{Kostant61}) is defined by
\[
	\mathcal{W}^{\rP_i} = \left\{w \in \mathcal{W} : w(\Phi^-) \cap \Phi^+ \subset \Delta(\mathfrak{u}_i)\right\}.
\]

Clearly $\W^{\rP_0}=\W$ and, by using the table, one can see that 
\begin{eqnarray}\label{eq:Weyl12} 
\W^{\rP_1}  = \left\{w_0, w_1, w_3, w_5\right\}\quad  \mbox{ and} \quad \W^{\rP_2}  =  \left\{ w_0, w_2, w_4, w_6 \right\} \,.
\end{eqnarray}

\subsection{Boundary of the Borel-Serre compactification}

In general, the boundary of the Borel-Serre compactification $\partial \rS_\Gamma = \overline{\rS}_\Gamma \setminus \rS_\Gamma$ is the union of subspaces indexed by the standard $\Q$-parabolic subgroups
\[
    \partial \rS_\Gamma =  \cup_{\rP \in \mathcal{P}_\Q(\Sp)} \partial_{\Gamma, \rP}.
\]
To simplify the notation we will denote by $\partial_{\rP}$ the space $\partial_{\Gamma, \mathrm{P}}$. In this case, as we are in the rank $2$ case, this covering defines a long exact sequence (of Mayer-Vietoris) in cohomology. We will use this long exact sequence to describe the cohomology of the boundary of the Borel-Serre compactification.


\section{Boundary Cohomology}\label{nerve}
In this section we make use of certain spectral sequences to obtain a description of the cohomology spaces $H^q(\partial_{\mathrm{P}}, \widetilde{\m})$ of the faces of the boundary associated to the parabolic subgroups of $\Sp$. Using this information, we give the details of the boundary cohomology $H^{\bullet}(\partial \rS_{\G}, \m_{\lambda})$. For more details the reader can see ~\cite{Harder2012} where boundary and Eisenstein cohomology has been discussed. The goal of this section is to present this result in a convenient way to describe the dimension of the Eisenstein cohomology and to be used in the Sections ~\ref{Dim cuspidal} and ~\ref{Dim Coho} together with the Euler characteristic. 

\subsection{The cohomology spaces $H^\bullet(\partial_{\mathrm{P}}, \tm)$}\label{Cohomology Parab}

For a parabolic $\rP$, let $\rM$ be its Levi quotient, $\mathrm{U}$ its unipotent radical and $\mathfrak{u}$ the Lie algebra of $\mathrm{U}(\R)$. We denote by $\pi_\rP : \rP \rightarrow \rM = \rP/\mathrm{U}$ the natural projection. We write $\Gamma^\mathrm{M} = \pi_\mathrm{P}(\Gamma \cap \mathrm{P(\R)})$ and $\rK_\infty^\mathrm{M} = \pi_\mathrm{P}(\rK_\infty \cap \mathrm{P}(\R))$. It is known that $\Gamma^\mathrm{M}, \rK_\infty^\mathrm{M} \subset \left.^\circ \mathrm{M}_\mathrm{P} (\mathbb{R}) \right.$ where for an algebraic group $\mathrm{G}$ over $\Q$ we denote by $X_\Q(\mathrm{G})$ the group of characters of $\mathrm{G}$ defined over $\Q$ and

$$ \left.^\circ \mathrm{G} \right. = \bigcap_{\chi \in X_\Q(\mathrm{G})} \chi^2. $$
One has a fibration 
\begin{equation}
\Gamma_\mathrm{U} \backslash \mathrm{U}(\R) \rightarrow \partial_\mathrm{P} \rightarrow \rS_{\Gamma}^\mathrm{M} \nonumber
\end{equation}

where we write $$ \rS_{\Gamma}^\mathrm{M} = \Gamma^\mathrm{M} \backslash \left.^\circ \rM(\R) \right. / \rK_\infty^\mathrm{M}$$

This fibration defines a spectral sequence in cohomology abutting to the cohomology of $\partial_\mathrm{P}$, and in this case it is well known that the spectral sequence degenerates in degree $2$ and gives a decomposition

\begin{equation}
    H^k(\partial_{\mathrm{P}}, \m) = \bigoplus_{p+q=k} H^p(\rS_{\Gamma}^\mathrm{M}, \widetilde{H^q(\mathfrak{u}, \m)}). \nonumber
\end{equation}

Finally, by using Kostant's theorem (see~\cite{Kostant61}) one can decompose $H^q(\mathfrak{u}, V)$ as a direct sum of irreducible representations of $\mathrm{M}$ indexed by some subset $\mathcal{W}^\mathrm{P} \subset \mathcal{W}$ of the Weyl group (the subset of Weyl representatives, see Section \ref{KostantRep}). One finally has

\begin{equation}\label{eq:hpm}
    H^k(\partial_{\mathrm{P}}, \tm) = \bigoplus_{w \in \mathcal{W}^\mathrm{P}} H^{k - \ell(w)}(\rS_{\Gamma}^\mathrm{M}, \tm_{w \cdot \lambda}) 
\end{equation}

where $\m_{w \cdot \lambda}$ denotes the irreducible representation of $\mathrm{M}$ with highest weight $w \cdot \lambda$. In the following subsections we will use \eqr{hpm} to describe the cohomology spaces $H^\bullet(\partial_{\mathrm{P}}, \tm)$ for each $\rP$.  For more details on this decomposition see \cite{Schwermer1994}.

In what follows, we will denote by $\mathrm{M}_i$ and $\partial_i$ the Levi quotient of the parabolic $\mathrm{P}_i$ and its corresponding subspace of the boundary, respectively.

\subsubsection{Cohomology of $\partial_0$} From Subsection~\ref{standard}, we know that $\rM_0 \cong \mathbb{G}_m \times \mathbb{G}_m$ and therefore $H^{q}(\rS_{\Gamma}^\mathrm{M_{0}}, \tm_{w \cdot \lambda}) = 0$ for $q > 0$. Following the decomposition ~\eqr{hpm} and  using the Table~\ref{table1}, we write:
\begin{align}
H^0(\partial_0, \tm_\lambda) & =  H^{0}(\rS_{\Gamma}^\mathrm{M_{0}}, \tm_{\lambda}) \nonumber\\
H^1(\partial_0, \tm_\lambda) & =  H^{0}(\rS_{\Gamma}^\mathrm{M_{0}}, \tm_{w_1 \cdot \lambda}) \oplus H^{0}(\rS_{\Gamma}^\mathrm{M_{0}}, \tm_{w_2\cdot\lambda}) \nonumber \\
H^2(\partial_0, \tm_\lambda) & =  H^{0}(\rS_{\Gamma}^\mathrm{M_{0}}, \tm_{w_3\cdot\lambda}) \oplus H^{0}(\rS_{\Gamma}^\mathrm{M_{0}}, \tm_{w_4\cdot \lambda}) \nonumber\\
H^3(\partial_0, \tm_\lambda) & =  H^{0}(\rS_{\Gamma}^\mathrm{M_{0}}, \tm_{w_5\cdot\lambda}) \oplus H^{0}(\rS_{\Gamma}^\mathrm{M_{0}}, \tm_{w_6\cdot\lambda}) \nonumber\\
H^4(\partial_0, \tm_\lambda) & =  H^{0}(\rS_{\Gamma}^\mathrm{M_{0}}, \tm_{w_7\cdot\lambda}) \,.\nonumber
\end{align}

The fact that
\begin{equation}
\left( \begin{array}{cccc}
 -1 &  &  & \\
 &  1  &  & \\
 &  &  1  & \\
 &  &  &  -1 \end{array}  \right), 
\left( \begin{array}{cccc}
 1  &  &  & \\
 & -1  &  & \\
 &  & -1  & \\
 &  &  &  1 \end{array}  \right) \in \rM_0(\Z) \nonumber
\end{equation}
has the effect that, if $w \cdot \lambda = n_1 \epsilon_1 + n_2 \epsilon_2$ then
\begin{equation}
H^{0}(\rS_{\Gamma}^\mathrm{M_{0}}, \tm_{w \cdot\lambda}) = \left\{\begin{array}{cccc}  
& \Q \,,  &  n_1, n_2 \mbox{ even} \\ 
&\\
& 0 \,, & \mr{otherwise}
\end{array}\qquad\,. \right . \nonumber
\end{equation}

Hence one has the following results. 
\begin{itemize}
    \item If $m_1$ and $m_2$ are even then
\begin{equation}
H^q(\partial_0, \tm_\lambda) =\left\{\begin{array}{cccc}  
& \mathbb{Q} \,,  &  q= 0, 1, 3, 4 \\ 
&\\
& 0 \,, & \mr{otherwise}
\end{array}\qquad\,, \right . \nonumber
\end{equation}

\item If $m_1$ is even and $m_2$ is odd then
\begin{equation}
H^q(\partial_0, \tm_\lambda) =\left\{\begin{array}{cccc}  
& \mathbb{Q} \,,  &  q= 1, 3 \\ 
&\\
& \mathbb{Q} \oplus \mathbb{Q} \,,  &  q= 2 \\ 
&\\
& 0 \,, & \mr{otherwise}
\end{array}\qquad\,, \right . \nonumber
\end{equation}
\item and finally, $H^q(\partial_0, \tm_\lambda)=0$ otherwise.
\end{itemize}

\subsubsection{Cohomology of $\partial_1$}\label{p1} 

From Subsection~\ref{standard}, we have $\rM_1 \cong \mathbb{G}_m \times \mr{SL}_2$. For an element $w \in \mathcal{W}^{\mathrm{P}_1}$, the highest weight $w \cdot \lambda$ written in Table~\ref{table1} as a pair $(a_w, b_w)$ means that we are working with the highest weight $a_w$ in $\mathrm{GL}_1$ and $b_w$ in $\mathrm{SL}_2$ while considering the usual fundamental weight in $\mathrm{SL}_2$. In this case $\mathrm{K}_\infty^{\rM_1} = \mathrm{SO}(2, \R)$ and 

\begin{equation} 
\left( \begin{array}{cccc}
 -1 &  &  & \\
 &  1  &  & \\
 &  &  1  & \\
 &  &  &  -1 \end{array}  \right), 
\left( \begin{array}{cccc}
 1  &  &  & \\
 & -1  &  & \\
 &  & -1  & \\
 &  &  &  1 \end{array}  \right) \in \rM_1(\Z) \cap Z(\mathrm{M}_1)\, \nonumber
\end{equation}

where $Z(\mathrm{M}_1)$ denotes the center of the Levi quotient $\mathrm{M}_1$.  Therefore, if 
\[
\rS^{\mathrm{SL}_2} = \mathrm{SL}_2(\Z) \backslash \mathrm{SL}_2(\R)/\mathrm{SO}(2, \R) \quad \mbox{ and } \quad w \cdot \lambda = a_w \epsilon_1 + b_w \epsilon_2
\]then 
\begin{equation}
H^q(\rS_{\Gamma}^\mathrm{M_1}, \tm_{w \cdot \lambda}) = \left\{\begin{array}{cccc}  
& H^q (\rS^{\mathrm{SL}_2}, \left. \tm_{w \cdot \lambda}\right|_{\mr{SL}_2(\Z)}) \,,  &  a_w, b_w \mbox{ even} \\ 
&\\
& 0 \,, & \mr{otherwise}
\end{array}\qquad\,, \right . \nonumber
\end{equation}

where $\left.\tm_{w \cdot \lambda}\right|_{\mr{SL}_2(\Z)}$ denotes the sheaf on $\rS^\mathrm{\mathrm{SL}_2}$ defined by the restriction of the representation $\m_{w \cdot \lambda}$ to the $\mathrm{SL}_2$ component of $\mathrm{M}_1$. By the Eichler-Shimura isomorphism one has

\begin{equation} \label{Eichler-Shimura}
H^q (\rS^{\mathrm{SL}_2}, \left. \tm_{w \cdot \lambda} \right|_{\mr{SL}_2(\Z)}) \otimes \C = \left\{\begin{array}{cccc}
& \mathcal{S}_{b_w+2} \oplus \overline{\mathcal{S}}_{b_w+2} \oplus Eis_{b_w+2} \,, &  q=1, b_w \neq 0 \\ 
&\\
& \mathbb{Q} \,,  &  q = b_w = 0 \\ 
&\\
& 0 \,, & \mr{otherwise}
\end{array}\qquad\,, \right . \nonumber
\end{equation}

where $\mathcal{S}_{b_w+2}$, $\overline{\mathcal{S}}_{b_w+2}$ denote the space of holomorphic and antiholomorphic cuspidal forms for $\mathrm{SL}_2(\Z)$ of weight $b_w+2$ and $Eis_{b_w+2}$ denotes the space of Eisenstein cohomology, which is isomorphic to the boundary cohomology of $\rS^{\mathrm{SL}_2}$.

Now, following~\eqr{hpm}, ~\eqr{Weyl12} and Table~\ref{table1}, we write:
\begin{eqnarray}
H^0(\partial_1, \tm_\lambda) & =&  H^0(\rS_{\Gamma}^\mathrm{M_1}, \tm_{\lambda}) \nonumber\\
H^1(\partial_1, \tm_\lambda) & =&  H^{1}(\rS_{\Gamma}^\mathrm{M_1}, \tm_{w_0 \cdot\lambda}) \oplus H^{0}(\rS_{\Gamma}^\mathrm{M_1}, \tm_{w_1 \cdot \lambda}) \nonumber \\
H^2(\partial_1, \tm_\lambda) & =& H^{1}(\rS_{\Gamma}^\mathrm{M_1}, \tm_{w_1 \cdot \lambda}) 
 \oplus H^{0}(\rS_{\Gamma}^\mathrm{M_1}, \tm_{w_3 \cdot \lambda}) \nonumber\\
H^3(\partial_1, \tm_\lambda) & = & H^{1}(\rS_{\Gamma}^\mathrm{M_1}, \tm_{w_3 \cdot \lambda}) \oplus H^{0}(\rS_{\Gamma}^\mathrm{M_1}, \tm_{w_5 \cdot \lambda}) \nonumber\\
H^4(\partial_1, \tm_\lambda) & = & H^{1}(\rS_{\Gamma}^\mathrm{M_1}, \tm_{w_5 \cdot \lambda}). \nonumber
\end{eqnarray}

In order to describe $H^\bullet(\partial_1, \tm_\lambda)$ one only has to consider the following four cases

\begin{itemize}
\item If $m_1$ is even and $m_2= 0$, then
\begin{equation}
H^q(\partial_1, \tm_\lambda) =\left\{\begin{array}{cccc}  
& H^0 (\rS_{\Gamma}^\mathrm{M_1}, \tm_{w_0 \cdot \lambda }) \cong \Q \,,  &  q = 0 \\ 
&\\
& H^0 (\rS_{\Gamma}^\mathrm{M_1}, \tm_{w_5 \cdot \lambda}) \cong \Q \,,  &  q = 3 \\ 
&\\
& 0 \,, & \mr{otherwise}
\end{array}\qquad\, \right.. \nonumber
\end{equation}

\item If $m_1$ and $m_2$ are even and $m_2 \neq 0$, then
\begin{equation}
H^q(\partial_1, \tm_\lambda) =\left\{\begin{array}{cccc}  
& H^1 (\rS_{\Gamma}^\mathrm{M_1}, \tm_{w_0 \cdot \lambda}) \,,  &  q = 1 \\ 
&\\
& H^1 (\rS_{\Gamma}^\mathrm{M_1}, \tm_{w_5 \cdot \lambda}) \,,  &  q = 4 \\ 
&\\
& 0 \,, & \mr{otherwise}
\end{array}\qquad\, \right.. \nonumber
\end{equation}

\item On the other hand, if $m_1$ is even and $m_2$ is odd then
\begin{equation}
H^q(\partial_1, \tm_\lambda) =\left\{\begin{array}{cccc}  
& H^1 (\rS_{\Gamma}^\mathrm{M_1}, \tm_{w_1 \cdot \lambda}) \,,  &  q = 2 \\ 
&\\
& H^1 (\rS_{\Gamma}^\mathrm{M_1}, \tm_{w_3 \cdot \lambda}) \,,  &  q = 3 \\ 
&\\
& 0 \,, & \mr{otherwise}
\end{array}\qquad\, \right.. \nonumber
\end{equation}

\item In all the other cases $H^\bullet(\partial_1, \tm_\lambda) = 0$.
\end{itemize}

\subsubsection{Cohomology of $\partial_2$}\label{p2}
Finally, in this case $\rM_2 \cong \mr{GL}_2$. In Table ~\ref{table1}, the element $w \cdot \lambda$ encoded by the pair $(a_w, b_w)$ means that we are working with the usual highest weight representation of $\mr{GL}_2$ associated with the character $a_w \varepsilon_1 + b_w \varepsilon_2$.  Therefore, $H^\bullet(\rS_{\Gamma}^\mathrm{M_2}, \tm_\lambda) = 0$ if $a_w + b_w$ is odd (because $-id_2$ will act as multiplication by $(-1)^{a_w+b_w}$). On the other hand, when $a_w = b_w$ is odd then the representation is one dimensional and 
\begin{equation}
\left( \begin{array}{cccc}
 -1 &  &  & \\
 &  1  &  & \\
 &  &  1  & \\
 &  &  &  -1 \end{array}  \right) \in \rM_2(\Z) \nonumber
\end{equation}
implies that $H^\bullet(\rS_{\Gamma}^\mathrm{M_2}, \tm_\lambda) = 0$. 

In this case, $\mathrm{K}_\infty^{\rM_2} = \mathrm{O}(2, \R)$. Now we suppose that $a_w + b_w$ is even. If
$$\rS^\mathrm{GL_2} = \mathrm{GL}_2(\Z) \backslash \mathrm{GL}_2(\R) / \mathrm{SO}(2, \R)$$
then 
\[
H^q (\rS^{\mathrm{GL}_2}, \tm_{w \cdot \lambda}) \cong H^q (\rS^{\mathrm{SL}_2}, \left. \tm_{w \cdot \lambda}\right|_{\mr{SL}_2(\Z)})
\]
and 
\begin{equation}
H^\bullet(\rS_{\Gamma}^\mathrm{M_2}, \tm_\lambda) = H^\bullet(\rS^\mathrm{GL_2}, \tm_\lambda)^{\mathrm{O}(2, \R)/\mathrm{SO}(2, \R)} \nonumber
\end{equation}
is the space of fixed points under the natural action of $\mathrm{O}(2, \R)/\mathrm{SO}(2, \R)$. 

Finally, one has
\begin{equation}
H^q (\rS_{\Gamma}^{\mathrm{M}_2}, \tm_{w \cdot \lambda}) \otimes \C = \left\{\begin{array}{cccc}
& \mathcal{S}_{a_w-b_w+2} \oplus Eis_{a_w, b_w} \,, &  q=1, a_w-b_w \neq 0 \\ 
&\\
& \mathbb{Q} \,,  &  q = a_w-b_w = 0 \\ 
&\\
& 0 \,, & \mr{otherwise}
\end{array}\qquad\,, \right . \nonumber
\end{equation}
where $Eis_{a_w, b_w}$ is the space of Eisenstein cohomology, isomorphic to the boundary cohomology of $\mathrm{GL}_2(\Z)$ with coefficients in the irreducible representation with highest weight $a_w \varepsilon_1 + b_w\varepsilon_2$.

Now, following~\eqr{Weyl12}, \eqr{hpm}, Table~\ref{table1} and the discussion carried out for $\rP_1$ in Subsection~\ref{p1}, we simply need to analyse the following spaces 
\begin{eqnarray}
H^0(\partial_2, \tm_\lambda) & =&  H^{0}(\rS_{\Gamma}^{\mathrm{M}_2}, \tm_{w_0 \cdot\lambda}) \nonumber\\
H^1(\partial_2, \tm_\lambda) & =&  H^{1}(\rS_{\Gamma}^{\mathrm{M}_2}, \tm_{w_0 \cdot\lambda}) \oplus H^{0}(\rS_{\Gamma}^{\mathrm{M}_2}, \tm_{w_2 \cdot \lambda}) \nonumber \\
H^2(\partial_2, \tm_\lambda) & =& H^{1}(\rS_{\Gamma}^{\mathrm{M}_2}, \tm_{w_2 \cdot \lambda}) \oplus H^{0}(\rS_{\Gamma}^{\mathrm{M}_2}, \tm_{w_4 \cdot \lambda}) \nonumber\\
H^3(\partial_2, \tm_\lambda) & = & H^{1}(\rS_{\Gamma}^{\mathrm{M}_2}, \tm_{w_4 \cdot \lambda}) \oplus H^{0}(\rS_{\Gamma}^{\mathrm{M}_2}, \tm_{w_6 \cdot \lambda}) \nonumber\\
H^4(\partial_2, \tm_\lambda) & = & H^{1}(\rS_{\Gamma}^{\mathrm{M}_2}, \tm_{w_6 \cdot \lambda}) \nonumber
\end{eqnarray}

and therefore, one can see that
\begin{itemize}

\item If $m_1 \not = 0$  is even then
\begin{equation}
H^q(\partial_2, \tm_\lambda) = \left\{\begin{array}{cccc}  
& H^1 (\rS_{\Gamma}^{\mathrm{M}_2}, \tm_{w_0 \cdot \lambda}) \,,  &  q = 1 \\ 
&\\
& H^1 (\rS_{\Gamma}^{\mathrm{M}_2}, \tm_{w_2 \cdot \lambda}) \,,  &  q = 2 \\ 
&\\
& H^1 (\rS_{\Gamma}^{\mathrm{M}_2}, \tm_{w_4 \cdot \lambda}) \,,  &  q = 3 \\ 
&\\
& H^1 (\rS_{\Gamma}^{\mathrm{M}_2}, \tm_{w_6 \cdot \lambda}) \,,  &  q = 4 \\ 
&\\
& 0 \,, & \mr{otherwise}
\end{array}\qquad\,, \right . \nonumber
\end{equation}
\item If $m_1 = 0$ and $m_2$ is even then
\begin{equation}
H^q(\partial_2, \m_\lambda) = \left\{\begin{array}{cccc}  
& H^0 (\rS_{\Gamma}^{\mathrm{M}_2}, \tm_{w_0 \cdot \lambda}) \,,  &  q = 0 \\ 
&\\
& H^1 (\rS_{\Gamma}^{\mathrm{M}_2}, \tm_{w_2 \cdot \lambda}) \,,  &  q = 2 \\ 
&\\
& H^1 (\rS_{\Gamma}^{\mathrm{M}_2}, \tm_{w_4 \cdot \lambda}) \,,  &  q = 3 \\ 
&\\
& 0 \,, & \mr{otherwise}
\end{array}\qquad\,, \right . \nonumber
\end{equation}
\item If $m_1 = 0$ and $m_2$ is odd then
\begin{equation}
H^q(\partial_2, \m_\lambda) = \left\{\begin{array}{cccc}  
& H^1 (\rS_{\Gamma}^{\mathrm{M}_2}, \tm_{w_2 \cdot \lambda}) \,,  &  q = 2 \\ 
&\\
& H^1 (\rS_{\Gamma}^{\mathrm{M}_2}, \tm_{w_4\cdot \lambda}) \oplus H^0 (\rS_{\Gamma}^{\mathrm{M}_2}, \tm_{w_6 \cdot \lambda})  \,,  &  q = 3 \\ 
&\\
& 0 \,, & \mr{otherwise}
\end{array}\qquad\,, \right . \nonumber
\end{equation}

\item Finally, if $m_1$  is odd then $H^q(\partial_2, \tm_\lambda) = 0$.

\end{itemize}

\subsection{Boundary Cohomology}\label{boundary}

In this subsection we use the results obtained in Subsection~\ref{Cohomology Parab} to describe the cohomology of the boundary. The covering of the boundary of the Borel-Serre compactification defines a spectral sequence in cohomology abutting to the cohomology of the boundary
\[
	E_1^{p, q} = \bigoplus_{prk(P)=(p+1)} H^q(\partial_\rP, \widetilde{\mathcal{M}}_\lambda) \Rightarrow H^{p+q}(\partial \rS_\Gamma, \widetilde{\mathcal{M}}_\lambda).
\]
where $prk(\rP)$ denotes the parabolic rank of $\rP$ (in this case $prk(\rP_1) = prk(\rP_2) = 1$ and $prk(\rP_0)=2$).
As the $\Q$-rank of $\Sp$ is $2$, this spectral sequence can be replaced by the long exact sequence
\begin{equation} 
\cdots \rightarrow H^{q-1}(\partial_0, \tm_\lambda) \rightarrow  H^q(\partial \rS_\Gamma, \tm_\lambda) \rightarrow H^q(\partial_1, \tm_\lambda) \oplus H^q(\partial_2, \tm_\lambda) \rightarrow H^q(\partial_0, \tm_\lambda) \rightarrow \cdots \nonumber
\end{equation} 

and one obtains the cohomology of the boundary by the description of the restriction morphisms $H^\bullet(\partial_i, \tm_\lambda) \rightarrow H^\bullet(\partial_0, \tm_\lambda)$. Finally, by using the results of the previous subsection, this reduces to the well known cases of $\mathrm{GL}_2(\Z)$ and $\mathrm{SL}_2(\Z)$.

Let $\lambda=m_1\lambda_1+m_2\lambda_2$ be the highest weight of the irreducible representation $\m_\lambda$. If $m_1$ is odd, then the fact that $-id_4 \in \Sp(\Z) \cap \mathrm{K}_\infty$ is an element of the center of $\Sp$ has the effect that $\widetilde{\m}_\lambda = 0$. Therefore we are only interested in the case $m_1$ even. This reduces to analyze in total six different subcases. We study these subcases by following a similar analysis as the one described in Section 4 of~\cite{BHHM2018}. As usual, we denote by $H^\bullet_!(\rS^{\mathrm{M}_i}, \tm)$ the inner cohomology (the kernel of the natural restriction to the cohomology of the boundary of the Borel-Serre compactification $r_i:H^\bullet(\rS_{\Gamma}^{\mathrm{M}_i}, \tm) \rightarrow H^\bullet(\partial \rS_{\Gamma}^{\mathrm{M}_i}, \tm)$ of $\rS^{\mathrm{M}_i}$). By using the calculation of the previous subsection, we now summarize the details of boundary cohomology in all the cases as follows:
\subsubsection{Case 1\,: $m_1 = 0$ and $m_2 = 0$ (trivial coefficient system)}
\begin{equation}
H^q(\partial \rS_\Gamma, \tm_\lambda) = \left\{\begin{array}{cccc}  
& \Q \,,  &  q = 0 \\
&\\
& \Q \oplus H^1_!(\rS_{\Gamma}^{\mathrm{M}_2}, \tm_{w_2 \cdot \lambda}) = \Q \,,  &  q = 2 \\ 
&\\
& \Q \oplus H^1_!(\rS_{\Gamma}^{\mathrm{M}_2}, \tm_{w_4 \cdot \lambda}) = \Q \,,  &  q = 3 \\ 
&\\
& \Q \,,  &  q = 5 \\
&\\
& 0 \,, & \mr{otherwise}
\end{array}\qquad\,, \right . \nonumber
\end{equation}

\subsubsection{Case 2\,: $m_1 = 0$ and $m_2 \not = 0$ even}
\begin{equation}
H^q(\partial \rS_\Gamma, \tm_\lambda) = \left\{\begin{array}{cccc}  
& H^1_! (\rS_{\Gamma}^{\mathrm{M}_1}, \tm_{w_0 \cdot \lambda}) \,,  &  q = 1 \\ 
&\\
& H^1_! (\rS_{\Gamma}^{\mathrm{M}_2}, \tm_{w_2 \cdot \lambda}) \,,  &  q = 2 \\ 
&\\
& H^1_! (\rS_{\Gamma}^{\mathrm{M}_2}, \tm_{w_4 \cdot \lambda}) \,,  &  q = 3 \\ 
&\\
& H^1_! (\rS_{\Gamma}^{\mathrm{M}_1}, \tm_{w_5 \cdot \lambda}) \,,  &  q = 4 \\ 
&\\
& 0 \,, & \mr{otherwise}
\end{array}\qquad\,, \right . \nonumber
\end{equation}

\subsubsection{Case 3\,: $m_1 \not = 0$ even and $m_2 = 0$}
\begin{equation}
H^q(\partial \rS_\Gamma, \tm_\lambda) = \left\{\begin{array}{cccc}  
& H^1_! (\rS_{\Gamma}^{\mathrm{M}_2}, \tm_{w_0 \cdot \lambda}) \,,  &  q = 1 \\ 
&\\
& H^1_! (\rS_{\Gamma}^{\mathrm{M}_2}, \tm_{w_2 \cdot \lambda}) \oplus \Q \,,  &  q = 2 \\ 
&\\
& H^1_! (\rS_{\Gamma}^{\mathrm{M}_2}, \tm_{w_4 \cdot \lambda}) \oplus \Q \,,  &  q = 3 \\ 
&\\
& H^1_! (\rS_{\Gamma}^{\mathrm{M}_2}, \tm_{w_6 \cdot \lambda}) \,,  &  q = 4 \\ 
&\\
& 0 \,, & \mr{otherwise}
\end{array}\qquad\,, \right . \nonumber
\end{equation}

\subsubsection{Case 4\,: $m_1 \not = 0$ even and $m_2 \not = 0$ even}
\begin{equation}
H^q(\partial \rS_\Gamma, \tm_\lambda) = \left\{\begin{array}{cccc}  
& \Q \oplus H^1_! (\rS_{\Gamma}^{\mathrm{M}_1}, \tm_{w_0 \cdot \lambda}) \oplus H^1_! (\rS_{\Gamma}^{\mathrm{M}_2}, \tm_{w_0 \cdot \lambda}) \,,  &  q = 1 \\ 
&\\
& H^1_! (\rS_{\Gamma}^{\mathrm{M}_2}, \tm_{w_2\cdot\lambda}) \,,  &  q = 2 \\ 
&\\
& H^1_! (\rS_{\Gamma}^{\mathrm{M}_2}, \tm_{w_4 \cdot \lambda}) \,,  &  q = 3 \\ 
&\\
& \Q \oplus H^1_! (\rS_{\Gamma}^{\mathrm{M}_1}, \tm_{w_5\cdot \lambda}) \oplus H^1_! (\rS_{\Gamma}^{\mathrm{M}_2}, \tm_{w_6 \cdot \lambda}) \,,  &  q = 4 \\ 
&\\
& 0 \,, & \mr{otherwise}
\end{array}\qquad\,, \right . \nonumber
\end{equation}

\subsubsection{Case 5\,: $m_1 = 0$ and $m_2$ odd}
\begin{equation}
H^q(\partial \rS_\Gamma, \tm_\lambda) = \left\{\begin{array}{cccc}  
& \Q \oplus H^1_! (\rS_{\Gamma}^{\mathrm{M}_1}, \tm_{w_1\cdot\lambda}) \oplus H^1_! (\rS_{\Gamma}^{\mathrm{M}_2}, \tm_{w_2\cdot\lambda}) \,,  &  q = 2 \\ 
&\\
& \Q \oplus H^1_! (\rS_{\Gamma}^{\mathrm{M}_1}, \tm_{w_3\cdot\lambda}) \oplus H^1_! (\rS_{\Gamma}^{\mathrm{M}_2}, \tm_{w_4\cdot\lambda}) \,,  &  q = 3 \\ 
&\\
& 0 \,, & \mr{otherwise}
\end{array}\qquad\,, \right . \nonumber
\end{equation}

\subsubsection{Case 6\,: $m_1 \not = 0$ even and $m_2$ odd}
\begin{equation}
H^q(\partial \rS_\Gamma, \tm_\lambda) = \left\{\begin{array}{cccc}  
& H^1_! (\rS_{\Gamma}^{\mathrm{M}_2}, \tm_{w_0\cdot \lambda}) \,,  &  q = 1 \\ 
&\\
& H^1_! (\rS_{\Gamma}^{\mathrm{M}_1}, \tm_{w_1\cdot\lambda}) \oplus H^1_! (\rM_{2}(\Z), \m_{w_2\cdot\lambda}) \,,  &  q = 2 \\ 
&\\
& H^1_! (\rS_{\Gamma}^{\mathrm{M}_1}, \tm_{w_3\cdot\lambda}) \oplus H^1_! (\rM_{2}(\Z), \m_{w_4\cdot\lambda}) \,,  &  q = 3 \\ 
&\\
& H^1_! (\rS_{\Gamma}^{\mathrm{M}_2}, \tm_{w_6\cdot\lambda}) \,,  &  q = 4 \\ 
&\\
& 0 \,, & \mr{otherwise}
\end{array}\qquad\,. \right . \nonumber
\end{equation}

\section{Eisenstein Cohomology}\label{Eisenstein}

The Eisenstein cohomology is known to be a powerful tool to study the natural restriction morphism $$r^q:H^q(\rS_\Gamma, \tm_\lambda) \rightarrow H^q(\partial \rS_\Gamma, \tm_\lambda)$$ to the cohomlogy of the boundary (see for example Theorem 4.11 of \cite{Schwermer83;LNM988}). In this section we give the description of the Eisesntein cohomology following the work of Harder in~\cite{Harder2012}. We summarize the details on Eisenstein cohomology with coefficients in $\m_\lambda$ which depends on the parity of the coefficients in the highest weight $\lambda = m_1\varpi_1 + m_2 \varpi_2$. As in the previous section, in what follows we denote by $\mathcal{S}_k$ the space of cuspidal forms of weight $k$ for $\mathrm{SL}_2(\mathbb{Z})$.

Let $\Sigma_{k}$ be the canonical basis of normalized eigenfunctions of $\mathcal{S}_k$. Then by using Eichler-Shimura isomorphism one can write
$$
H^1_!(\rS_{\Gamma}^{\mathrm{M}_2}, \tm_{w \cdot \lambda} \otimes \C) = \oplus_{f \in \Sigma_k} H^1_!(\rS_{\Gamma}^{\mathrm{M}_2}, \tm_{w \cdot \lambda} \otimes \C)(f),$$

where the $\C$-vector spaces $H^1_!(\rS_{\Gamma}^{\mathrm{M}_2}, \tm_{w \cdot \lambda} \otimes \C)(f)$ are one dimensional.
We denote $\mathcal{Z}_k = \left\{f \in \Sigma_k \mid L(f, \frac{k}{2}) \neq 0) \right\}$, then one knows:

\begin{thm}[\cite{Harder2012}]
If $m_1$ or $m_2$ is not even, then the restriction morphism in degrees 3 and 4 defines an isomorphism from Eisenstein cohomology to Boundary cohomology, and Eisenstein cohomology is $0$ in degrees 1 and 2. If $m_1 = m_2 = 0$ then the Eisenstein cohomology is one dimensional in degree $0$ and degree $2$ and it is $0$ in the other degrees. Finally, if $m_1 = 0$ and $m_2$ is even, then 
\begin{equation}
H_{Eis}^q(\rS_\Gamma, \tm_\lambda \otimes \C) = \left\{\begin{array}{cccc}  
& \oplus _{f \in \mathcal{Z}_{2m_2+4}} H^1_!(\rS_{\Gamma}^{\mathrm{M}_2}, \tm_{w_2 \cdot \lambda} \otimes \C)(f) \,,  &  q = 2 \\ 
&\\
& \oplus _{f \notin \mathcal{Z}_{2m_2+4}} H^1_!(\rS_{\Gamma}^{\mathrm{M}_2}, \tm_{w_4 \cdot \lambda} \otimes \C)(f)  \,,  &  q = 3 \\ 
&\\
& H^1_!(\rS_{\Gamma}^{\mathrm{M}_1}, \tm_{w_5 \cdot \lambda} \otimes \C) \,,  &  q = 4 \\ 
&\\
& 0 \,, & \mr{otherwise}
\end{array}\qquad\,. \right . \nonumber
\end{equation}
\end{thm}

By using the calculations of the previous section we get the following.

\subsubsection{If $m_1 = m_2 = 0$,}

\begin{equation} 
H_{Eis}^q(\rS_\Gamma, \tm_\lambda \otimes \C) = \left\{\begin{array}{cccc}  
& \C \,,  &  q = 0, 2 \\
&\\
& 0 \,, & \mr{otherwise}
\end{array}\qquad\,, \right . \nonumber
\end{equation}

\subsubsection{If $m_1 = 0$ and $m_2 \not = 0$ is even,} 

\begin{equation} 
H_{Eis}^q(\rS_\Gamma, \tm_\lambda \otimes \C) = \left\{\begin{array}{cccc}  
& \oplus _{f \in \mathcal{Z}_{2m_2+4}} H^1_!(\rS_{\Gamma}^{\mathrm{M}_2}, \tm_{w_2 \cdot \lambda} \otimes \C)(f) \,,  &  q = 2 \\ 
&\\
& \oplus _{f \notin \mathcal{Z}_{2m_2+4}} H^1_!(\rS_{\Gamma}^{\mathrm{M}_2}, \tm_{w_4 \cdot \lambda} \otimes \C)(f)  \,,  &  q = 3 \\ 
&\\
& \mathcal{S}_{m_2+2} \oplus \overline{\mathcal{S}}_{m_2 + 2} \,,  &  q = 4 \\ 
&\\
& 0 \,, & \mr{otherwise}
\end{array}\qquad\,, \right . \nonumber
\end{equation}

\subsubsection{If $m_1 \not = 0$ is even and $m_2 = 0$,}

\begin{equation}
H_{Eis}^q(\rS_\Gamma, \tm_\lambda \otimes \C) = \left\{\begin{array}{cccc}  
& \mathcal{S}_{m_1 + 4} \oplus \mathbb{C} \,,  &  q = 3 \\ 
&\\
& \mathcal{S}_{m_1 + 2} \,,  &  q = 4 \\ 
&\\
& 0 \,, & \mr{otherwise}
\end{array}\qquad\,, \right . \nonumber
\end{equation}

\subsubsection{If $m_1 \not = 0$ is even and $m_2 \not = 0$ is even,}

\begin{equation}
H_{Eis}^q(\rS_\Gamma, \tm_\lambda \otimes \C) = \left\{\begin{array}{cccc}  
& \mathcal{S}_{m_1 + 2m_2 + 4} \,,  &  q = 3 \\ 
&\\
& \mathbb{Q} \oplus \mathcal{S}_{m_2 + 2} \oplus \overline{\mathcal{S}}_{m_2 + 2} \oplus \mathcal{S}_{m_1 + 2} \,,  &  q = 4 \\ 
&\\
& 0 \,, & \mr{otherwise}
\end{array}\qquad\,, \right . \nonumber
\end{equation}

\subsubsection{If $m_1 = 0$ and $m_2$ is odd,}

\begin{equation}
H_{Eis}^q(\rS_\Gamma, \tm_\lambda \otimes \C) = \left\{\begin{array}{cccc}  
& \mathbb{Q} \oplus \mathcal{S}_{m_2 + 3} \oplus \overline{\mathcal{S}}_{m_2 + 3} \oplus \mathcal{S}_{2m_2 + 4} \,,  &  q = 3 \\ 
&\\
& 0 \,, & \mr{otherwise}
\end{array}\qquad\,, \right . \nonumber
\end{equation}

\subsubsection{If $m_1 \not = 0$ is even and $m_2$ is odd,}

\begin{equation}
H_{Eis}^q(\rS_\Gamma, \tm_\lambda \otimes \C) = \left\{\begin{array}{cccc}  
& \mathcal{S}_{m_1 + m_2 + 3} \oplus \overline{\mathcal{S}}_{m_1 + m_2 + 3} \oplus \mathcal{S}_{m_1 + 2m_2 + 4} \,,  &  q = 3 \\ 
&\\
& \mathcal{S}_{m_1 + 2}  \,,  &  q = 4 \\ 
&\\
& 0 \,, & \mr{otherwise}
\end{array}\qquad\,. \right . \nonumber
\end{equation}   


\section{Torsion Elements and Orbifold Euler Characteristics}\label{Euler char}

In this section, the orbifold Euler characteristics of the centralizers of torsion elements of $\Sp(\Z)$ are being calculated. Euler characteristic has been a useful tool to address the various problems in group cohomology. For example see~\cite{Horozov2014}. We quickly review the basics about Euler characteristic.  The homological Euler characteristic $\chi_h$ of a group $\Gamma$ with coefficients in a representation $V$ is defined by 
\begin{equation}\label{hec}
\chi_h(\Gamma,V)=\sum_{i=0}^{\infty}\, (-1)^{i} \, \mr{dim} \, H^{i}(\Gamma, V). \nonumber
\end{equation} 

For an arithmetic group $\Gamma_1$, let $\Gamma_1'$ be a torsion free finite index subgroup of $\Gamma_1$ (one knows that every arithmetic group of rank greater than one contains a torsion free finite index subgroup).  Then the orbifold Euler characteristic of $\Gamma_1$ is given by $$\chi_{orb}(\Gamma_1)= [\Gamma_1 : \Gamma_1']^{-1} \chi_h(\Gamma_1').$$

From now on, orbifold Euler characteristic will be simply denoted by $\chi$. Note that, if $\Gamma_1$ is torsion free then $\chi_{h}(\Gamma_1, V) = \chi(\Gamma_1, V).$

The following properties of $\chi$ will be very handy in the forthcoming discussion.
 \begin{itemize}
\item Let $\Gamma_0$, $\Gamma_1$ and  $\Gamma_2$ be groups such that 
$ 1 \longrightarrow \Gamma_1 \longrightarrow \Gamma_0 \longrightarrow \Gamma_2 \longrightarrow 1$ is exact then $\chi(\Gamma_0)= \chi(\Gamma_1) \chi(\Gamma_2)$.
\item If $\Gamma_0$ is finite of order $\left| \Gamma_0 \right|$ then $\chi(\Gamma_0)=\frac{1}{\left| \Gamma_0 \right|}$.
 \end{itemize}

We have introduced the orbifold Euler characteristic in order to use the following formula. If $\Gamma_1$ has torsion elements then we make use of the following result (see~\cite{Horozov2005}).
\begin{equation}\label{eq:hecT}
\chi_h(\Gamma,V)=\sum_{(T)} \chi_{orb}(C(T)) Tr(T^{-1}, V)
\end{equation}
where the sum runs over the set of representatives of conjugacy classes in $\Gamma_1$ of torsion elements $T$ of $\Gamma_1$ and $C(T)$ denotes the centralizer of $T$ in $\Gamma_1$.

In this section, we make use of the following lemma. 

\begin{lema}
$\chi(\mr{Sp}_4(\Z)) = \zeta(-1)\zeta(-3)= -\frac{1}{1440}$ and $\chi(\mr{SL}_2(\Z)) = \zeta(-1)= -\frac{1}{12}$.
\end{lema}

The proof of the above two identities follows from a quiet well known fact that for $\Gamma =\mr{Sp}_{2g}(\Z)$  $$\chi(\Gamma) = \prod_{k=1}^{g} \zeta(1-2k),$$ where $\zeta$ denotes the Riemann's zeta function. The above formula follows from the work of Harder in~\cite{Harder71}. For a quick reference on the appearance of this formula see Theorem 5 on Page 344 of~\cite{SerreBourbaki71} and Example $(iii)$ on Page 158 of~\cite{Serre71}.

Following~\eqr{hecT}, we know that in order to compute $\chi_{h}(\mr{Sp}_{4}(\Z), V)$, we need the list of the conjugacy classes of all  torsion elements. 

We continue by giving the list of representatives of the conjugacy classes of torsion elements of the group $\Sp(\Z)$. For that we need to introduce some notation. 
As before, for each $n \in \mathbb{N}$, $id_n \in \mathrm{GL}_n(\mathbb{Z})$ denotes the identity matrix, 

\begin{equation}
U = \left( \begin{array}{rr}
1 & 0 \\
1 & -1 \end{array}  \right), \qquad
W = \left( \begin{array}{cc}
0 & -1 \\
1 & -1 \end{array}  \right), \qquad
J_2 = \left( \begin{array}{rr}
0 & -1 \\
1 & 0 \end{array}  \right) \nonumber,\\  
\end{equation}
\begin{equation}
R= \left(\begin{array}{rrrr}
0 & 0 & -1 & 0 \\
0 & 0 & 0 & -1 \\
1 & 0 & 0 & 1 \\
0 & 1 & 1 & 0 \end{array}  \right),\quad 
S= \left(\begin{array}{rrrr}
0 & 1 & 0 & 0 \\
0 & 0 & -1 & 0 \\
0 & 0 & -1 & 1 \\
1 & 1 & -1 & 0 \end{array}  \right), \quad
T= \left( \begin{array}{rrrr}
0 & -1 & 1 & 0 \\
-1 & 0 & 1 & 1 \\
-1 & 1 & 0 & 0 \\
0 & -1 & 0 & 0 \end{array}  \right)\,. \nonumber
\end{equation} \\
Let us consider the following operations between $2 \times 2$ matrices in $\mathrm{M}_2(\Z)$:
\begin{equation}
\left( \begin{array}{cc}
a_1 & b_1 \\
c_1 & d_1 \end{array}  \right) \ast
\left(\begin{array}{cc}
a_2 & b_2 \\
c_2 & d_2 \end{array}  \right)=
\left( \begin{array}{cccc}
a_1 & 0 & b_1 & 0 \\
0 & a_2 & 0 & b_2 \\
c_1 & 0 & d_1 & 0 \\
0 & c_2 & 0 & d_2 \end{array}  \right), \nonumber \\
\end{equation} 

\begin{equation}
\left( \begin{array}{cc}
a_1 & b_1 \\
c_1 & d_1 \end{array}  \right) \dotplus
\left( \begin{array}{cc}
a_2 & b_2 \\
c_2 & d_2 \end{array}  \right) =
\left( \begin{array}{cccc}
a_1 & b_1 & 0 & 0 \\
c_1 & d_1 & 0 & 0 \\
0 & 0 & a_2 & b_2 \\
0 & 0 & c_2 & d_2 \end{array}  \right), \nonumber \\
\end{equation}
and
\begin{equation}
\left( \begin{array}{cc}
a_1 & b_1 \\
c_1 & d_1 \end{array}  \right) \circ
\left(\begin{array}{cc}
a_2 & b_2 \\
c_2 & d_2 \end{array}  \right) =
\left( \begin{array}{cccc}
0 & a_1 & 0 & b_1 \\
a_2 & 0 & b_2 & 0 \\
0 & c_1 & 0 & d_1 \\
c_2 & 0 & d_2 & 0 \end{array}  \right)\,.
 \nonumber \\
\end{equation}

Then one has the following theorem (see ~\cite{Yang}).

\begin{thm}[Yang]
A complete list of representatives of the conjugacy classes of torsion elements in $\Sp(\Z)$ is given below in the table, where $\Phi_n$ denotes the $n$-th cyclotomic polynomial.

{ \begin{center}
\scriptsize\renewcommand{\arraystretch}{2}
\begin{longtable}{|c|c|c|||c|c|c|c|c|c|}
\hline
Torsion Element  &  Expression   &  Characteristic polynomial & Torsion Element  &  Expression   &  Characteristic polynomial    \\
\hline
\hline
$T_1$ & $id_4$  & $\Phi_{_1}^{^4}$  & $T_2$ & $-id_4$ & $\Phi_{_2}^{^4}$ \\  
\hline
$T_3$ & $id_2 \ast -id_2$  &$\Phi_{_1}^{^2} \Phi_{_2}^{^2}$   & $T_4$ & $U \dotplus U^t$  & $\Phi_{_1}^{^2} \Phi_{_2}^{^2}$\\
\hline
$T_5$ & $W \ast W$ & $\Phi_{_3}^{^2}$   & $T_6$& $W^t \ast W^t$  & $\Phi_{_3}^{^2}$\\
\hline
$T_7$ & $W \ast W^t$ &$\Phi_{_3}^{^2}$  & $T_8$ & $id_2 \ast W$ & $\Phi_{_1}^{^2} \Phi_{_3}$ \\
\hline
$T_9$ & $id_2 \ast W^t$ &$\Phi_{_1}^{^2} \Phi_{_3}$ & $T_{10}$ & $J_2 \ast J_2$ & $\Phi_{_4}^{^2}$\\
\hline
$T_{11}$ & $-(J_2 \ast J_2)$ & $\Phi_{_4}^{^2}$ &$T_{12}$ & $J_2 \ast (-J_2)$  & $\Phi_{_4}^{^2}$\\
\hline
$T_{13}$ & $(-id_2) \circ id_2$ & $\Phi_{_4}^{^2}$ &  $T_{14}$&$id_2 \ast J_2$  & $\Phi_{_1}^{^2} \Phi_{_4}$ \\
\hline
$T_{15}$ & $id_2 \ast (-J_2)$ &$\Phi_{_1}^{^2} \Phi_{_4}$ & $T_{16}$ & $(-id_2) \ast J_2$  & $\Phi_{_2}^{^2} \Phi_{_4}$ \\
\hline
$T_{17}$ & $-(id_2 \ast J_2)$ & $\Phi_{_2}^{^2} \Phi_{_4}$ & $T_{18}$& $S$ & $ \Phi_{_5}$\\
\hline
$T_{19}$ & $S^2$&$\Phi_{_5}$ &$T_{20}$& $S^3$ & $\Phi_{_5}$ \\
\hline
$T_{21}$ &$S^4$& $\Phi_{_5}$ &$T_{22}$ & $-(W \ast W)$  & $\Phi_{_6}^{^2}$  \\
\hline
$T_{23}$ & $-(W^t \ast W^t)$ &  $\Phi_{_6}^{^2}$ & $T_{24}$ & $-(W \ast W^t)$ & $\Phi_{_6}^{^2}$ \\
\hline
$T_{25}$& $id_2 \ast (-W)$  & $\Phi_{_1}^{^2} \Phi_{_6}$  &$T_{26}$ & $id_2 \ast (-W^t)$ & $\Phi_{_1}^{^2} \Phi_{_6}$  \\
\hline
$T_{27}$ & $-(id_2 \ast W)$  &$\Phi_{_2}^{^2} \Phi_{_6}$ &$T_{28}$ & $-(id_2 \ast W^t)$ & $\Phi_{_2}^{^2} \Phi_{_6}$\\
\hline
$T_{29}$ & $(-id_2) \ast W$  & $\Phi_{_2}^{^2} \Phi_{_3}$ &$T_{30}$ & $(-id_2) \ast W^t$ & $\Phi_{_2}^{^2} \Phi_{_3}$ \\
\hline
$T_{31}$ &$W \ast (-W)$ & $\Phi_{_3} \Phi_{_6}$ &$T_{32}$ & $W \ast (-W^t)$  &  $\Phi_{_3} \Phi_{_6}$ \\
\hline
$T_{33}$ & $W^t \ast (-W)$  & $\Phi_{_3} \Phi_{_6}$ &$T_{34}$ & $W^t \ast (-W^t)$  & $\Phi_{_3} \Phi_{_6}$ \\
\hline
$T_{35}$ &$id_2 \circ W$   & $\Phi_{_3} \Phi_{_6}$  &$T_{36}$ & $id_2 \circ W^t$  &  $\Phi_{_3} \Phi_{_6}$\\
\hline
$T_{37}$ &$R$   & $\Phi_{_3} \Phi_{_6}$ &$T_{38}$ & $-R$   & $\Phi_{_3} \Phi_{_6}$ \\
\hline
$T_{39}$ & $id_2 \circ J_2$  &$\Phi_{_8}$  &$T_{40}$ & $id_2 \circ (-J_2)$ &$\Phi_{_8}$ \\
\hline
$T_{41}$ & $T$ &$\Phi_{_8}$&$T_{42}$& $-T$ &  $\Phi_{_8}$ \\
\hline
$T_{43}$ & $-S$ & $\Phi_{_{10}}$&$T_{44}$ & $-S^2$  & $\Phi_{_{10}}$ \\
\hline
$T_{45}$ & $-S^3$ & $\Phi_{_{10}}$ &$T_{46}$ &$-S^4$  & $\Phi_{_{10}}$ \\
\hline
$T_{47}$ &$id_2 \circ (-W)$  &$\Phi_{_{12}}$ &$T_{48}$ & $id_2 \circ (-W^t)$ &$\Phi_{_{12}}$  \\
\hline
$T_{49}$ & $J_2 \ast W$ & $\Phi_{_3} \Phi_{_4}$ &$T_{50}$ & $J_2 \ast W^t$  & $\Phi_{_3} \Phi_{_4}$\\
\hline
$T_{51}$ &$J_2^t \ast W$  &$\Phi_{_3} \Phi_{_4}$ &$T_{52}$ & $J_2^t \ast W^t$ & $\Phi_{_3} \Phi_{_4}$\\
\hline
$T_{53}$ & $J_2 \ast (-W)$ & $\Phi_{_4} \Phi_{_6}$&$T_{54}$ &  $J_2 \ast (-W^t)$& $\Phi_{_4} \Phi_{_6}$ \\
\hline
$T_{55}$ & $J_2^t \ast (-W)$ & $\Phi_{_4} \Phi_{_6}$& $T_{56}$ & $J_2^t \ast (-W^t)$ & $\Phi_{_4} \Phi_{_6}$\\
\hline
\hline
\caption{Torsion Elements}\label{torsionelements}
\end{longtable}
\end{center}
}
\end{thm}

To study the centralizer of each element $T_i$ and eventually their Euler characterisics, we considered the equations coming from the symplectic identity $g^{t} J g =J$, and the relation $g T_i=T_i g$ for $g \in  C(T_i) \subseteq \Sp(\Z)$. 
We now give the details case by case. Let us fix an element $$g = \left( \begin {array}{cccc} a_1 & a_2 & b_1 & b_2\\ 
a_3 & a_4 & b_3 &b_4\\ 
c_1 & c_2 & d_1& d_2\\ 
c_3 & c_4 & d_3 & d_4\end {array} \right).$$

\subsection{Centralizer and Euler characteristic of  $T_1$ and $T_2$}  Here $T_1 = id_4$ and $T_2 =-id_4$. Clearly, $C(T_1) =C (T_2) =\mr{Sp}_4(\Z)$. Therefore  $$\chi(C(T_1)) = \chi(C(T_2)) = -\frac{1}{1440}.$$

\subsection{Centralizer and Euler characteristic of  $T_4$} \label{T4} From the definition $$ T_4 = U \dotplus U^{t} = \left( \begin {array}{cccc} 1&0&0&0\\ 1&-1&0&0\\ 0&0&1&1\\ 0&0&0&-1\end {array} \right) \,.$$  For any  $g \in C(T_4)$, solving $g T_4 =T_4 g$ we  obtain that $g$ must be of the form
\begin{equation}
g = \left( \begin{array}{cccc}
2a_3 + a_4 & 0 & 2b_3 & b_3 \\
a_3 & a_4 & b_3 & b_4 \\
c_1 & c_3 & 2d_2 + d_4 & d_2 \\
c_3 & -2c_3 & 0 & d_4 \end{array}  \right). \nonumber \\
\end{equation}
Solving $g^t J g =J$, we immediately get
$$ det \tmt{2a_3 + a_4}{b_3}{2c_1 + c_3}{2d_2 + d_4} =1, \quad det \tmt{a_4}{(b_3-2b_4)}{c_3}{d_4} =1\,.$$  This establishes a map between $C(T_4)$ and $\mathrm{SL}_2(\Z) \times \mathrm{SL}_2(\Z)$ given by
\begin{equation}
g:=g(a_3, a_4, b_3, b_4, c_1, c_3, d_2, d_4) \mapsto \left( \left( \begin{array}{cc}
2a_3 + a_4 & b_3 \\
2c_1 + c_3 & 2d_2 + d_4 \end{array}  \right) ,
\left( \begin{array}{cc}
a_4 & b_3-2b_4 \\
c_3 & d_4 \end{array}  \right) \right)\,,
 \nonumber \\
\end{equation}
which is an isomorphism between $C(T_4)$ and the index $6$ subgroup $$H=\left\{(A, B) \in \mathrm{SL}_2(\Z) \times \mathrm{SL}_2(\Z) \mid A \equiv B \, (\mr{mod}\, 2) \right\}$$ of $\mathrm{SL}_2(\Z) \times \mathrm{SL}_2(\Z)$. Therefore $\chi(C(T_4))= 6 \chi(\mathrm{SL}_2(\Z))^2 = \frac{1}{24}$.\\

\subsection{Centralizer and Euler characteristic of  $T_5$, $T_6$, $T_{22}$ and $T_{23}$} \label{Centralizer 5} From the definition, it is clear that $C(T_5) = C(T_{22})$ and $C(T_6) = C(T_{23})$.  Let us consider $$T_5 =W\ast W = \left( \begin {array}{cccc} 0&0&-1&0\\ 0&0&0&-1\\ 1&0&-1&0\\ 0&1&0&-1\end {array} \right)\,.$$

Solving $T_5 \, g = g \, T_5$ gives $ g =\tmt{x}{y}{-y}{x+y} $, where $x=\tmt{a_1}{a_2}{a_3}{a_4}$ and $y=\tmt{b_1}{b_2}{b_3}{b_4}$ and using this with $g^{t} J g=J$ gives the following relations
\begin{eqnarray*}
& a_1^2 + a_1 b_1 + b_1^2 + a_3^2 +a_3 b_3 +b_3^2 =1 \,, \quad & a_2^2 + a_2 b_2 + b_2^2 + a_4^2 +a_4 b_4 +b_4^2 =1 \,,\\
& a_1 a_2 + b_1 b_2 + a_3 a_4 + b_3 b_4 + a_1 b_2 + a_3 b_4=0\,, \quad & a_1 b_2 - a_2 b_1 -a_3 b_4 + a_4 b_3 =0 \,.
\end{eqnarray*} 
Consider $\omega  = \frac{1}{2} + i \frac{\sqrt{3}}{2}= e^{\frac{2\pi i}{3}}$. Following the above relations we can show that $x - {\omega} y \in \mathrm{U}_2(\Z[{\omega}])=\left\{ u=A+ {\omega} B | A, B \in \mr{M}_2(\R),  u^{*} u = id_2 \right\}$ (where $\mr{M}_2(\R)$ denotes the space of $2$ by $2$ matrices with coefficients in $\R$). This establishes an isomorphism between $C(T_5)$ and $\mathrm{U}_2(\Z[{\omega}])$. One can also see from those relations, that one of the pairs $(a_1, b_1), (a_3, b_3)$ must be zero and the other must be
$(1, 0), (0, 1), (1, -1), (-1, 0), (0, -1)$ or $(-1, 1)$. The same is true for the pairs $(a_2, b_2), (a_4, b_4)$. Even more, $(a_1, b_1)$ is nonzero if and only if the pair $(a_4, b_4)$ is nonzero. 
$C(T_5)$ has therefore order $72$ and $\chi(C(T_5)) = \chi(C(T_{22}))=\frac{1}{72}$.

Similarly, as $T_6 = T_5^{t}$,  $C(T_6)$ is isomorphic to $C(T_5)$ and $\chi(C(T_6))=\chi(C(T_{23}))=\frac{1}{72}.$

\subsection{Centralizer and Euler characteristic of  $T_7$ and $T_{24}$}  $T_{24} =- T_7$, therefore $C(T_7) = C(T_{24})$. From the definition $$T_7 = \left( \begin {array}{cccc} 0&0&-1&0\\ 0&0&0&1\\ 1&0&-1&0\\ 0&-1&0&-1\end {array} \right).$$ By using the fact that $T_7 g = g T_7$, we get $$ g=\left( \begin {array}{cccc} a_1 &  a_2 & b_1 & b_2 \\ a_3 & a_4 & b_3 & b_4\\ - b_1& b_2& a_1 +  b_1& b_2- a_2 \\ b_3&-b_4&-a_3-b_3&a_4- b_4\end {array} \right)\,.$$ Using the equation $g^{t} J g =J$, gives us the following conditions 
\begin{eqnarray*}
& (a_1^2 + a_1 b_1 + b_1^2)-(a_3^{2} + a_3 b_3 + b_3^2)=1 \\
& (a_4^2 - a_4 b_4 + b_4^2)-(a_2^{2} - a_2 b_2 + b_2^2)=1 \\
& a_1 b_2 + a_2b_1 -a_3b_4 -a_4b_3 = 0 \\
& a_1a_2 - b_1b_2 - a_3a_4 + b_3b_4 -a_1b_2 +a_3b_4 = 0
\end{eqnarray*}

and these equations show that $$\left( \begin {array}{cc} a_1 - b_1 \omega &  a_2 + b_2 \omega \\ a_3 - b_3 \omega & a_4 + b_4 \omega \end {array} \right) \in \mathrm{U}_{(1, 1)} (\mathbb{Z}[\omega]) =\left\{ u=A+ {\omega} B | A, B \in \mr{M}_2(\R),  u^{*}I_{1, -1} u = I_{1, -1} \right\}$$

where $I_{1, -1} = \left( \begin {array}{cc} 1 &  0  \\ 0  & - 1 \end {array} \right)$ and $\omega = e^{\frac{2\pi i}{3}}$. Even more, one can see that this establishes an isomorphism between $C(T_7)$ and $\mathrm{U}_{(1, 1)}(\mathbb{Z}[\omega])$.

\begin{prop}
$\mathrm{SU}_{(1, 1)} (\mathbb{Z}[\omega])$ is isomorphic to a subgroup of $\mathrm{SL}_2(\mathbb{Z})$ of index $4$ and is a subgroup of index $6$ of $\mathrm{U}_{(1, 1)}(\mathbb{Z}[\omega])$. In particular $\chi(C(T_7)) = \frac{4}{6} \chi(\mathrm{SL}_2(\mathbb{Z})) = -\frac{1}{18}$.
\end{prop}

\begin{proof}
We take $D = \frac{1}{\omega^2 - 1} \left( \begin {array}{cc} -1 &  \omega  \\ -\omega  &  1  \end {array} \right)$. Let $g$ be an element of $\mathrm{SU}_{(1, 1)} (\mathbb{Z}[\omega])$. Then $\exists$ $a, b, c, d \in \mathbb{Z}$ such that $$g = \left( \begin {array}{cc} a + b\omega &  c+d \omega  \\ c + d \bar{\omega}  &  a + b \bar{\omega}  \end {array} \right).$$
Hence $D g D^{-1} = \left( \begin {array}{cc} a - d  &  b - c  \\ d - b - c  &  a + d - b  \end {array} \right)$ and one can see that conjugation by $D$ gives an isomorphism between $\mathrm{SU}_{(1, 1)}(\mathbb{Z}[\omega])$ and
$$\left\{\tmt{m_1}{m_2}{m_3}{m_4} \in \mathrm{SL}_2(\mathbb{Z}) \mid (m_4 - m_1) - 2(m_3-m_2) \equiv 2(m_4 - m_1) - (m_3-m_2) \equiv 0 \mbox{ mod 3} \right\}$$
which is a subgroup of index $4$ in $\mathrm{SL}_2(\mathbb{Z})$.

\end{proof}

\subsection{Centralizer and Euler characteristic of  $T_{10}$ and $T_{11}$} Note that $C(T_{10}) = C(T_{11})$. For any $g \in C(T_{11})$, solving the equations $g^{t} J g=J$ and $g T_{11} = T_{11} g$ gives us that   $g=\left( \begin {array}{cccc} 
x & y\\
-y & x \end {array} \right) $  where $x  =\left( \begin {array}{cccc} 
a_1& a_2\\ 
a_3& a_4 \end {array} \right)$ and $y = \left( \begin {array}{cccc} 
b_1& b_2\\ 
b_3& b_4 \end {array} \right) $, and 
\begin{eqnarray*}
& a_1 b_2 - a_2 b_1 + a_3 b_4 - a_4 b_3  =  0\,, & a_1 a_2 + a_3 a_4 + b_1 b_2 + b_3 b_4 = 0 \,, \\ \nonumber
 & a_{1}^{2}  + a_{3}^{2} + b_{1}^{2} + b_{3}^{2} = 1 \,, &  a_{2}^{2}  + a_{4}^{2} + b_{2}^{2} + b_{4}^{2}  =  1 \,.\\ \nonumber
\end{eqnarray*}
These equations are equivalent to the fact that $u = x+i y$ is an element of $\mathrm{U}_{2}(\Z[i])$. This defines an isomorphism between $C(T_{11})$ and $\mathrm{U}_{2}(\Z[i])$. By a similar argument as the one used in Subsection \ref{Centralizer 5} one can see that $C(T_{11})$ has $32$ elements and  therefore $$\chi(C(T_{10})) = \chi(C(T_{11})) = \frac{1}{32}.$$

\subsection{Centralizer and Euler characteristic of  $T_{12}$} By definition $$T_{12}= J_2 \ast (-J_2) = \left( \begin {array}{cccc} 0&0&-1&0\\ 0&0&0&1
\\ 1&0&0&0\\ 0&-1&0&0\end {array}\right)\,.$$  Let $g \in C(T_{12})$. As $ T_{12} \, g = g T_{12}$ we get $$g = \left(\begin {array}{cccc}  
a_1 &a_2 & b_1& b_2 \\ 
a_3 &a_4 &b_3 &b_4  \\ 
-b_1& b_2&a_1 &- a_2\\ 
b_3 &-b_4&-a_3& a_4\end {array}\right) $$

and by the fact that $g^{t} J g =J$ we obtain the following relations 
\begin{eqnarray*}
& a_1 b_2 +a_2 b_1 -a_3 b_4 - a_4 b_3 =0 \,, \quad & a_1 a_2 - a_3 a_4 -b_1 b_2 + b_3 b_4 =0\,,\\
& a_1^2 + b_ 1^2 - a_3^2 - b_3^2 = 1\,, \quad & a_4^2 + b_ 4^2 - a_2^2 - b_2^2 = 1\,.
\end{eqnarray*}
This can also be written in terms of dot product by 
\begin{align} 
& (a_1, b_1).(b_2, a_2) = (a_3, b_3).(b_4, a_4)  \,, \quad & (a_1, b_1).(a_2, -b_2) = (a_3, b_3).(a_4, -b_4),\nonumber\\
& \left|\left|(a_1, b_1)\right|\right|^2 - \left|\left|(a_3, b_3)\right|\right|^2 = 1\,, \quad & \left|\left|(a_4, b_4)\right|\right|^2 - \left|\left|(a_2, b_2)\right|\right|^2 = 1\,, \nonumber
\end{align} 

which shows that $\left( \begin {array}{cc} a_1 - b_1i & a_2 + b_2 i \\ a_3 - b_3 i& a_4 +b_4i
\end {array}\right) \in \mathrm{U}_{(1,1)}(\mathbb{Z}[i])$. In fact this describes an isomorphism between $C(T_{12})$ and $\mathrm{U}_{(1,1)}(\mathbb{Z}[i])$).
 
\begin{prop}
$\mathrm{SU}_{(1,1)}(\mathbb{Z}[i])$ is isomorphic to an index $3$ subgroup of $\mathrm{SL}_2(\mathbb{Z})$ and is a subgroup of index $4$ in $\mathrm{U}_{(1, 1)}(\mathbb{Z}[i])$. In particular, $\chi(C(T_{12})) = \chi(\mathrm{U}_{(1, 1)}(\mathbb{Z}[i])) = -\frac{1}{16}$.
\end{prop}

\begin{proof}
Consider the function $f:\mathrm{SU}_{(1,1)}(\mathbb{Z}[i]) \rightarrow \mathrm{SL}_2(\mathbb{Z})$ defined as follows: If $g \in \mathrm{SU}_{(1, 1)}(\mathbb{Z}[i])$ then there exist $a, b, c, d \in \mathbb{Z}$ such that 
$g = \left(\begin {array}{cc}  a + b i & c + d i \\ c -  d i & a - b i \end {array}\right)$ and we define
\[
f\left(\begin {array}{cc}  
a + b i & c + d i \\ 
c - d i & a - b i \end {array}\right) =  \left(\begin {array}{cc}  
a + c & d - b \\ 
d + b & a - c\end {array}\right).
\]
The fact that $f(g) \in \mathrm{SL}_2(\mathbb{Z})$ follows from the fact that $a^2 + b^2 - c^2 - d^2 = 1$, and one can see that $f$ is, even more, a morphism of groups. It is also clear that $f$ is injective and therefore an isomorphism into its image, which is the subgrpup of $\mathrm{SL}_2(\mathbb{Z})$ of all the matrices that can be written as
\[
\left(\begin {array}{cc}  
a + c & d - b \\ 
d + b & a - c \end {array}\right)
\]
with $a, b, c, d \in \mathbb{Z}$. This is exactly the subgroup of all the matrices $\left(\begin {array}{cc}  m_1 & m_2 \\ m_3 & m_4 \end {array}\right) \in \mathrm{SL}_2(\mathbb{Z})$ satisfying $m_1 \equiv m_4 \quad mod (2)$ and $m_2 \equiv m_3 \quad mod (2)$, which is a subgroup of $\mathrm{SL}_2(\mathbb{Z})$ of index $3$.
\end{proof}

\subsection{Centralizer and Euler characteristic of  $T_{13}$} By definition $$T_{13} = \left( \begin {array}{cccc} 0&-1&0&0\\ 1&0&0&0
\\ 0&0&0&-1\\ 0&0&1&0\end {array} \right).$$ For any $g\in C(T_{13})$ after solving $g^t J g =J$ and $g T_{13} =T_{13 } g$, we obtain that $g$ can be written as $$g=\left( \begin {array}{cccc} 
a_1& a_2&b_1& b_2\\ 
- a_2& a_1&-b_2& b_1\\ 
c_1&c_2& d_1&d_2\\ 
-c_2 & c_1 & - d_2& d_1\end {array} \right) $$ where the coefficients satisfy the following equations
\begin{align}
 &a_1 c_2 = a_2 c_1,\quad b_1 d_2 = b_2 d_1, \label{Property 13} \\ 
 &a_1 d_2 -a_2 d_1 + b_1 c_2 - b_2 c_1 = 0 , \label{Property 13, ii} \\ 
 &a_1 d_1 + a_2 d_2 - b_1 c_1 - b_2 c_2 =1. \label{Property 13, iii}
\end{align}

We consider the following elements $z_a = a_1 + ia_2, z_b = b_1 + ib_2, z_c = c_1 + ic_2, z_d = d_1 + id_2 \in \mathbb{Z}[i]$. One can identify $C(T_{13})$ with a subgroup of $\mathrm{GL}_2(\mathbb{Z}[i])$ by the map $f$ that sends $g$ to the matrix 
$$f(g) = \left( \begin {array}{cc} 
z_a & z_b \\ 
z_c & z_d \end {array} \right) .$$

Equations (\ref{Property 13}), (\ref{Property 13, ii}) and (\ref{Property 13, iii}) could be written as
\begin{equation} \label{Property 13 iv}
z_a \bar{z}_c, z_b \bar{z}_d \in \mathbb{Z} \mbox{ and } z_a\bar{z}_d - \bar{z}_bz_c = 1. \nonumber
\end{equation}  

In particular, if $x=z_a\bar{z}_d, y=\bar{z}_bz_c$ then $x, y \in \mathbb{Z}[i]$ satisfy $x = y + 1$ and $x \bar{y} = z_a \bar{z}_c z_b \bar{z}_d \in \mathbb{Z}$. Hence $(y+1)\bar{y} \in \mathbb{Z}$ and this implies immediately that $y \in \mathbb{Z}$ and therefore $x \in \mathbb{Z}$.

We have proved
\begin{equation} 
z_a \bar{z}_c, z_b \bar{z}_d, z_a\bar{z}_d, \bar{z}_bz_c \in \mathbb{Z} \nonumber
\end{equation} 
or equivalently
\begin{equation} \label{Property 13 v}
a_1c_2 = a_2c_1, b_1d_2 = b_2d_1, a_1d_2 = a_2d_1, b_1c_2 =b_2c_1.\nonumber
\end{equation} 

Now it follows a study case by case.

\begin{itemize}
\item[] If $z_a = 0$, then $z_a\bar{z}_d - \bar{z}_bz_c = 1$ implies $- \bar{z}_bz_c = 1$. Therefore $z_b \in \left\{1, -1, i, -i\right\}$. If $z_b = 1$ or $-1$, then $z_c = -z_b$ and $z_d \in \mathbb{Z}$ (because $z_b \bar{z}_d \in \mathbb{Z}$). If $z_b = i$ or $-i$, then $z_c = z_b$ and as $z_b \bar{z}_d \in \mathbb{Z}$, $z_d \in i \mathbb{Z}$. In both cases, by using the fact that $z_a\bar{z}_d - \bar{z}_bz_c = 1$, one obtains that $f(g) \in \mathrm{SL}_2(\mathbb{Z}) \cup i \mathrm{SL}_2(\mathbb{Z})$.

\item[] If $z_a \neq 0$, then suppose $a_2 \neq 0$. We know, by (\ref{Property 13, iii}), that $a_1 d_1 + a_2 d_2 - b_1 c_1 - b_2 c_2 = 1$ but this can be written as
\begin{align}
a_1 d_1 + a_2 d_2 - b_1 c_1 - b_2 c_2 &= 1 \nonumber \\
\frac{a_2 ^2 a_1 d_1 - a_2^2 b_1 c_1}{a_2^2} + a_2 d_2 - b_2 c_2 &= 1 \nonumber \\
\frac{a_2 a_1^2 d_2 - a_2 b_1 a_1 c_2}{a_2^2} + a_2 d_2 - b_2 c_2 &= 1 \nonumber \\
\frac{a_1 ^2 a_2 d_2 - a_2 a_1 b_2 c_1}{a_2^2} + a_2 d_2 - b_2 c_2 &= 1 \nonumber \\
\frac{a_1 ^2 a_2 d_2 - a_1^2 b_2 c_2}{a_2^2} + a_2 d_2 - b_2 c_2 &= 1 \nonumber \\
\frac{a_1 ^2 + a_2^2}{a_2^2}(a_2 d_2 - b_2 c_2) &= 1 \nonumber
\end{align}
and this implies $a_1 = 0$. By using $z_a \bar{z}_c, z_b \bar{z}_d, z_a\bar{z}_d, \bar{z}_bz_c \in \mathbb{Z}$ and $z_a\bar{z}_d - \bar{z}_bz_c = 1$ one has $f(g) \in i \mathrm{SL}_2(\mathbb{Z})$. If $a_2 = 0$ then one can similarly see that $f(g) \in \mathrm{SL}_2(\mathbb{Z})$.

We have proved that in all the cases $f(g) \in \mathrm{SL}_2(\mathbb{Z}) \cup i \mathrm{SL}_2(\mathbb{Z})$. It is clear that $f$ is injective and each element in $\mathrm{SL}_2(\mathbb{Z}) \cup i \mathrm{SL}_2(\mathbb{Z})$ corresponds to an element in $C(T_{13})$. Therefore $f:C(T_{13}) \rightarrow \mathrm{SL}_2(\mathbb{Z}) \cup i \mathrm{SL}_2(\mathbb{Z})$ is an isomorphism and $\chi(C(T_{13})) = \frac{\zeta(-1)}{2} = -\frac{1}{24}$.
\end{itemize}

\subsection{Centralizer and Euler characteristic of  $T_{18}$, $T_{19},$ $T_{20},$ $T_{21},$ $T_{43},$   $T_{44},$ $ T_{45}, $ and $T_{46}$} Here  $$ T_{18}= S= \left( \begin {array}{cccc} 0&1&0&0\\ 0&0&-1&0\\ 0&0&-1&1\\ 1&1&-1&0\end {array} \right)  \,, \qquad S^{5} = id_4 \,.$$  Also, $T_{43} =- S$ and 
\begin{eqnarray*} 
& T_{19} = S^{2}, \quad  & T_{20} = S^{3}, \qquad  T_{21}= S^{4} \,,\\
& T_{44} = - S^{2}, \quad  & T_{45} = - S^{3}, \qquad  T_{46}=- S^{4}  \,.
\end{eqnarray*}

Therefore they all have the same centralizer and it is enough to calculate just one of them. Let  $g \in C(S)$ then solving $g S = S g$ one can see that $g$ is of the form  
$$ g =  \left( \begin {array}{cccc} 
a & b & c & d \\ 
d & a + d &-b-c-d & c \\ 
-c & -c-d & a -b & b+c+d \\ 
b+d & b-c & -b & a+c+d \end {array} \right)$$   

and the fact that $g^{t} J g=J$ implies, among other conditions, that 
\begin{align}
 1 &= 2cd+2d^2+c^2+b^2+db+a^2+ca+2da \nonumber \\
 &= \left(\frac{a}{2} + d + c\right)^2 + \left(\frac{d}{2} + b\right)^2 + \left(\frac{d}{\sqrt{2}} + \frac{a}{\sqrt{2}}\right)^2 + \left(\frac{d}{2}\right)^2 + \left(\frac{a}{2}\right)^2 \nonumber.
\end{align}

Following these conditions we get, through a case by case study, that $$C(T_{18}) = \left\{id_4, S, S^2, S^3, S^4, -id_4, -S, -S^2, -S^3, -S^4 \right\}.$$ Hence $$\chi(C(T_{n})) = \frac{1}{10}, \quad \mr{for} \quad n \in \left\{18, 19, 20, 21, 43, 44,45,46 \right\}\,.$$ 

\subsection{Centralizer and Euler characteristic of $T_{35}, T_{36}, T_{37}, T_{38}, T_{47}  \, \mr{and} \, T_{48}$ }  Consider
$$ T_{35} = id_2 \circ W = \left( \begin {array}{cccc} 0&1&0&0\\ 0&0&-1&0
\\ 0&0&0&1\\ 1&0&-1&0\end {array}\right)\,. $$
Let $g \in C(T_{35})$, then as $g \, T_{35} = T_{35}\, g$ we get that $g$ has the form $$ g = \left( \begin {array}{cccc} a_1 &a_2&b_1&b_2\\ b_2&a_1&-a_2-b_2& b_1
\\ -b_1&-b_2&a_1+b_1&a_2+b_2\\ a_2+b_2&-b_1&-a_2 & a_1+b_1\end {array}\right)\,.$$

Using the fact that $g^{t} J g =J$ we get, among other conditions, that
$$ a_1^2 + a_1 b_1 + b_1^2  + a_2^2 + a_2 b_2 + b_2^2=1 $$

and by using that condition one can verify that $C(T_{35})$ is the group of $12$ elements $\Z_6 \times \Z_2$ generated by $T_{35}$ and $-id_4$.
Hence, $$\chi(C(T_{35})) = \frac{1}{12}\,.$$

Using the same procedure one can show that $$\chi(C(T_{36})) = \chi(C(T_{37})) = \chi(C(T_{38})) = \chi(C(T_{47})) = \chi(C(T_{48})) = \frac{1}{12}\,,$$ the only difference is that for $k \in \left\{47, 48\right\}$, $C(T_k)$ is the group $\Z_{12}$ generated by $T_k$.

\subsection{Centralizer and Euler characteristic of $T_{39}, T_{40}$}   Let $g$ be an element of $C(T_{39})$. By the fact that  $g T_{39} =T_{39} g$, we get that $g$ is of the form $\tmt{x}{y}{-y}{x}$ where $x =\tmt{a_1}{a_2}{b_2}{a_1}$ and $y = \tmt{b_1}{b_2}{-a_2}{b_1}$. As $g^{t} J g =J$, one can see that
$$a_1^{2}+ a_2^{2}+ b_2^{2}+ b_1^{2} =1$$

and from that condition one can finally verify that $C(T_{39})$ is the group $\Z_8$ generated by $T_{39}$. Following a similar procedure, one can show that the same holds for $C(T_{40})$. Therefore,
$$\chi(C(T_{39})) = \chi(C(T_{40}))  = \frac{1}{8}.$$

\subsection{Centralizer and Euler characteristic of $T_{41}, T_{42}$} We have $ T_{41} = T $ and $T_{42} = -T$, so $C(T_{41}) = C(T_{42})$. Let $ g \in C(T_{41})$. As $g T_{41} = T_{41} g$  one can see that $g$ has the form $$g=\left(\begin {array}{cccc} a_1& a_2& b_1& b_2 \\ -2 b_1+2 b_2- a_2 & a_1+ 
b_1-b_2+a_2 &-a_2 & 2 b_1 -  b_2 + a_2 \\ -3 b_1+2 b_2 -2 a_2 &-2 b_2
+2 b_1+a_2& a_1&2 a_2+2 b_1- b_2 \\ 2 a_2+2b_1- b_2&-2 b_1+
b_2- a_2&b_2& b_2- a_2+a_1-b_1
\end {array} \right)$$

and $g^{t} J g =J$ implies, among other conditions, that
$$-2b_2a_2-4b_2b_1+2b_2^2+2a_2^2+4b_1a_2+3b_1^2+a_1^2 = 1$$
and
$$2b_2a_2+3b_2b_1-b_2^2-a_2^2-3b_1a_2-2b_1^2+a_1a_2+a_1b_2 = 0.$$
Therefore
\begin{align*}
1 &= -b_2b_1+b_2^2+a_2^2+b_1a_2+b_1^2+a_1^2+a_1a_2+a_1b_2 \nonumber \\
  &= \left(\frac{a_1 + a_2}{\sqrt{2}}\right)^2 + \left(\frac{a_2 + b_1}{\sqrt{2}}\right)^2 + \left(\frac{b_1 - b_2}{\sqrt{2}}\right)^2 + \left(\frac{b_2 + a_1}{\sqrt{2}}\right)^2 .\nonumber
\end{align*}
Analyzing this and considering all possible choices one can see that $C(T_{41})$ is the group $\Z_8$ generated by $T_{41}$ and
$$ \chi(C(T_{41})) = \chi(C(T_{42})) = \frac{1}{8}.$$

\subsection{Orbifold Euler characteristics}

\par For the remaining cases of torsion elements $T$, one can obtain $\chi(C(T))$ by using the same ideas. Thus we decided to spare the reader the tedious calculations and we list the results in the following table.

\begin{thm}
The following is a complete list of the orbifold Euler characteristics of the centralizers of each representative of the conjugacy classes of torsion elements in $\Sp(\mathbb{Z})$.

{ \begin{center}
\scriptsize\renewcommand{\arraystretch}{2}
\begin{longtable}{|c|c|c|c|c|c|c|c|c|} 
\hline
Case &Torsion Element $T$ & $C(T)$ &  Euler Characteristic $\chi(C(T))$ \\
\hline
\hline
A & $T_1, T_2$ & $\Sp(\Z)$ & $-\frac{1}{1440}$ \\
\hline
B &$T_3$ & $\mr{SL}_2(\Z) \times \mr{SL}_2(\Z)$ & $\frac{1}{144}$ \\
\hline
C &$T_4$ & $H \subset \mr{SL}_2(\Z) \times \mr{SL}_2(\Z)$  (see  ~\ref{T4}) & $\frac{6}{144}$\\
\hline
D &$T_5, T_6, T_{22}, T_{23}$ & $\mr{U}_2(\Z[w])$ & $\frac{1}{72}$\\
\hline
E &$T_7, T_{24}$ & $\mr{U}_{(1, 1)}(\Z[w])$ & $-\frac{1}{18}$  \\
\hline
F &$T_8, T_9, T_{25}, T_{26},T_{27}, T_{28},T_{29},T_{30} $ & $\mr{SL}_2(\Z) \times \Z_6$ & $-\frac{1}{72}$ \\
\hline
G &$T_{10}, T_{11}$ & $\mr{U}_2(\Z[i])$ & $\frac{1}{32}$\\
\hline
H &$T_{12}$ & $\mr{U}_{(1, 1)}(\Z[i])$ & $-\frac{1}{16}$ \\
\hline
I &$T_{13}$ & $\mr{SL}_2(\Z) \rtimes \Z_2$ & $-\frac{1}{24}$\\
\hline
J &$T_{14},  T_{15}, T_{16}, T_{17}$ & $\mr{SL}_2(\Z) \times \Z_4$ & $-\frac{1}{48}$\\
\hline
K &$T_{18}, T_{19}, T_{20}, T_{21}, T_{43}, T_{44}, T_{45}, T_{46}$ & $\Z_{10}$ & $\frac{1}{10}$\\
\hline
L &$T_{31}, T_{32}, T_{33}, T_{34}$ & $\Z_6 \times \Z_6$ & $\frac{1}{36}$\\
\hline
M &$T_{35}, T_{36}, T_{37}, T_{38}$ and $T_{47},T_{48}$  & $\Z_6 \times \Z_2 \mbox{ and } \Z_{12}$ & $\frac{1}{12}$\\
\hline
N &$T_{39}, T_{40}, T_{41}, T_{42}$ & $\Z_8$ & $\frac{1}{8}$\\
\hline
O &$T_{49}, T_{50}, T_{51}, T_{52}, T_{53},T_{54},T_{55},T_{56}$ & $\Z_6 \times \Z_4$ & $\frac{1}{24}$\\
\hline
\hline
\hline
\caption{Orbifold Euler characteristics of the centralizers of torsion elements}\label{orbifoldeuler}
\end{longtable}
\end{center}
}
\end{thm}


\section{Traces of Torsion Elements}\label{Traces}

In this section we calculate the traces of the torsion elements of $\Gamma$ with respect to every finite dimensional irreducible representation of $\mathrm{Sp}_4$ with highest weight $\lambda$. At first we only calculate this trace with respect to the symmetric powers $Sym^n V$ of the standard representation $V$ of $\mathrm{Sp}_4$ (i.e. when $\lambda =  n \lambda_1$). We will strongly use the fact that the symmetric power representation of $\mathrm{Sp}_4$ is the restriction to $\mathrm{Sp}_4$ of the symmetric power of the standard representation of $\mathrm{GL}_4$. Therefore, to calculate the trace of a torsion element $T \in \Gamma$ it suffices to calculate the trace of a conjugate element of $T$ in $\mathrm{GL}_4(\mathbb{C})$ with respect to the symmetric power representation of $\mathrm{GL}_4$. We make use of the Gelfand-Cetlin basis of the finite dimensional highest weight representations of $\mathrm{GL}_4$ (see \cite{Gelfand50}). The aforementioned basis $\left\{ \xi_\Lambda \right\}$ is indexed by ``patterns" $\Lambda$, that in the special case of the symmetric power representations $Sym^n V$, can be described as triplets of numbers $(a, b, c) \in \mathbb{Z}$ satisfying $n \geq a \geq b \geq c \geq 0$. We will therefore denote the corresponding element of the basis as $\Lambda_{a, b, c}$. With this notation one has
\begin{align}
    E_{1, 1} \cdot \Lambda_{a, b, c} &= c\Lambda_{a, b, c}, \nonumber \\
    E_{2, 2} \cdot \Lambda_{a, b, c} &= (b-c)\Lambda_{a, b, c},\nonumber \\ E_{3, 3} \cdot \Lambda_{a, b, c} &= (a-b)\Lambda_{a, b, c}, \nonumber \\
    E_{4, 4} \cdot \Lambda_{a, b, c} &= (n-a)\Lambda_{a, b, c}, \nonumber
\end{align}
where $E_{ii} \in \mathfrak{gl}_4(\mathbb{C})$ denotes the diagonal matrix whose entries are all $0$ except for a $1$ in the $i$-th place.

We obtain the traces with respect to any other highest weight by using the following result comming from the theory of the Weyl character forumula:

\begin{thm} \label{generalize}
For two nonnegative integers $n_1 \geq n_2$, let $\lambda$ be $(n_1-n_2) \lambda_1 + n_2\lambda_2$ and $H_{n_1, n_2}(T)$ denote the trace $Tr(T, \m_\lambda)$. For a nonnegative integer $n$, $H_{n}(T)$ denotes $Tr(T, Sym^{n}V)$. Then one has $$H_{n_1, n_2}(T) = H_{n_1}(T)(H_{n_2}(T) + H_{n_2-2}(T)) -  (H_{n_1+1}(T)+H_{n_1-1}(T))H_{n_2-1}(T),$$
where, in the case $n < 0$ one uses the convention $H_n(T) = 0$.
\end{thm}

(See \cite{FultonHarris}, Proposition 24.22, 
page 406, for the proof of this theorem).

\subsection{Traces with respect to $Sym^{n}V$} \label{Traces Sym}

In this section we will denote by $diag(d_1, d_2, d_3, d_4)$ the corresponding $4 \times 4$ diagonal matrix. For an element $v$ of the $n^{th}$ symmetric power $Sym^n(V)$ of the standard representation $V$ of $\mathrm{GL}_4$, we will denote by $X.v$ the action of an element $X \in \mathfrak{gl}_4(\mathbb{C})$ and by $\rho_n(g)(v)$ the action of an element $g \in \mathrm{GL}_4(\mathbb{C})$. Therefore one has the following commutative diagram
$$\xymatrixcolsep{6pc}\xymatrix{
\mr{Sp}_4(\C) \ar@{<-}^{exp}[d]  \ar@{->}^{\rho_{n}}[r] & \mr{GL}(\m_\lambda) \ar@{<-}^{exp}[d] \\
\mathfrak{sp}_4(\C)  \ar@{->}^{d \rho_{n}}[r] & \mathfrak{gl}(\m_\lambda) }$$

In this section, for each $k \in \N$, $\xi_k$ will denote a $k$-th primitive root of unity. 

\subsubsection{Traces of $T_1$ and $T_2$} $ $

As $T_1=id_4$ and $T_2 = -id_4$, it is clear that $$Tr(T_1, Sym^n V) = dim(Sym^n V) = \frac{1}{6}(n+1)(n+2)(n+3)$$ and $Tr(T_2, Sym^n V) = (-1)^n dim(Sym^n V) = (-1)^n \frac{1}{6}(n+1)(n+2)(n+3).$

\subsubsection{Traces of $T_3$ and $T_4$} $ $

Let $T$ be $T_3$ or $T_4$. Then $T^{-1}$ is conjugate to the diagonal matrix $diag(1, 1, -1, -1)$, and therefore it suffices to study $Tr(diag(1, 1, -1, -1), Sym^n(V))$. One has $E_{44} . \Lambda_{a, b, c} = (n - a) \Lambda_{a, b, c}$ and $E_{33}.\Lambda_{a, b, c} = (a-b) \Lambda_{a, b, c}$. Therefore, as $diag(1, 1, -1, -1) =  exp(i \pi  (E_{33} + E_{44}))$, one has $\rho_n(diag(1, 1, -1, -1))(\Lambda_{a, b, c}) = (-1)^{n-b} \Lambda_{a, b, c}$ and 
\begin{align}
Tr(T^{-1}, Sym^n(V)) &= \sum_{a=0}^n \sum_{b=0}^a \sum_{c=0}^b  (-1)^{n-b} = (-1)^{n} \sum_{a=0}^n \sum_{b=0}^a (-1)^{b}(b+1). \nonumber
\end{align}
$\sum_{b=0}^a (-1)^{b}(b+1)$ is equal to $-\frac{a+1}{2}$ if $a$ is odd and is equal to $\frac{a}{2} + 1$ if $a$ is even. Now, if $n$ is odd one has
\begin{align}
Tr(T^{-1}, Sym^n(V)) &= (-1)^{n} \sum_{a=0}^n \sum_{b=0}^a (-1)^{b}(b+1) = (-1) \sum_{a=0}^\frac{n+1}{2} \left( (a+1) - \frac{2a+2}{2}\right)  = 0 \nonumber 
\end{align}
while, on the other hand, if $n$ is even one has
\begin{align}
Tr(T^{-1}, Sym^n(V)) &= \sum_{a=0}^{n} \sum_{b=0}^a (-1)^{b}(b+1) =  \sum_{b=0}^n (-1)^{b}(b+1) = \frac{n}{2} + 1. \nonumber 
\end{align}

Whence
$Tr(T^{-1}, Sym^n(V)) = 
\left\{ 
\begin{tabular}{rl}
$\frac{n}{2}+1,$ & $n \equiv 0 \mod 2$\\
$0,$ & $n \equiv 1 \mod 2$\\
\end{tabular}
\right. $.

\subsubsection{Traces of $T_5$, $T_6$, $T_7$, $T_{22}$, $T_{23}$ and $T_{24}$} $ $

$T_5$, $T_6$ and $T_7$ are conjugate to the diagonal matrix $T = diag(\xi_3, \xi_3, \xi_3^2, \xi_3^2)$. As $T_{22} = -T_{5}$, $T_{23} = -T_{6}$ and $T_{24} = -T_{7}$, one has 
$$ Tr(T_{22}^{-1}, Sym^n(V)) = Tr(T_{23}^{-1}, Sym^n(V)) = Tr(T_{24}^{-1}, Sym^n(V)) = (-1)^n Tr(T^{-1}, Sym^n(V)).$$
We will therefore calculate the trace $Tr(T^{-1}, Sym^n(V))$. We have $$(2E_{11} + 2E_{22} + E_{33} + E_{44}). \Lambda_{a, b, c} = (n+b) \Lambda_{a, b, c}$$ and therefore $\rho_n(T^{-1})(\Lambda_{a, b, c}) = \xi_3^{n+b} \Lambda_{a, b, c}$. Now, 

\begin{align}
Tr(T^{-1}, Sym^n(V)) &= \sum_{a=0}^n \sum_{b=0}^a \sum_{c=0}^b  \xi_3^{n+b} = \xi_3^{n} \sum_{a=0}^n \sum_{b=0}^a \xi_3^{b}(b+1). \nonumber
\end{align}

If $a$ is congruent to $2$ modulo $3$, then 
\begin{align}
\sum_{b=0}^a \xi_3^{b}(b+1)  &= \sum_{b=0}^{\frac{a-2}{3}} \xi_3^{3b}(3b+1) + \xi_3^{3b+1}(3b+2) + \xi_3^{3b+2}(3b+3) \nonumber \\
&= \sum_{b=0}^{\frac{a-2}{3}} \xi_3^{3b+1} + 2\xi_3^{3b+2} = \frac{a+1}{3} (\xi_3 + 2\xi_3^{2}). \nonumber
\end{align}
One can see that if $a$ is congruent to $0$ modulo $3$, then $\sum_{b=0}^a \xi_3^{-b}(b+1) =  \frac{a}{3} (\xi_3^{2} + 2) + 1$ and, if $a$ is congruent to $1$ modulo $3$, one has $\sum_{b=0}^a \xi_3^{-b}(b+1) = \frac{a+2}{3}(1 + 2 \xi_3)$.

Therefore,
$$Tr(T^{-1}, Sym^n(V)) = 
\left\{ 
\begin{tabular}{rl}
$\frac{n+3}{3},$ & $n \equiv 0 \mod 3$\\
$-\frac{2n+4}{3},$ & $n \equiv 1 \mod 3$\\
$\frac{n+1}{3},$ & $n \equiv 2 \mod 3$\\
\end{tabular}
\right. .$$

The traces of the remaining elements can be calculated with the same ideas, thus we decided to spare the reader most of the tedious details by writing a brief description of the calculations that lead to the results.

\subsubsection{Traces of $T_8$, $T_9$, $T_{27}$ and $T_{28}$} $ $

$T_8$, $T_9$ are conjugate to the diagonal matrix $T = diag(1, 1, \xi_3^2, \xi_3)$. Again, as $T_{27} = -T_{8}$, $T_{28} = -T_{9}$, one has 
$$ Tr(T_{27}^{-1}, Sym^n(V)) = Tr(T_{28}^{-1}, Sym^n(V)) = (-1)^n Tr(T^{-1}, Sym^n(V)).$$
We will therefore calculate the trace $Tr(T^{-1}, Sym^n(V))$. Note that $$(E_{33} + 2E_{44}). \Lambda_{a, b, c} = (2n-a-b) \Lambda_{a, b, c}$$ and therefore $\rho_n(T^{-1})(\Lambda_{a, b, c}) = \xi_3^{2n-a-b} \Lambda_{a, b, c}$. Now

\begin{align}
Tr(T^{-1}, Sym^n(V)) &= \sum_{a=0}^n \sum_{b=0}^a \sum_{c=0}^b  \xi_3^{2n-a-b} = \xi_3^{2n} \sum_{a=0}^n \xi_3^{-a} \sum_{b=0}^a \xi_3^{-b}(b+1). \nonumber
\end{align}

Hence,
$$Tr(T^{-1}, Sym^n (V)) = 
\left\{ 
\begin{tabular}{rl}
$\frac{n+3}{3},$ & $n \equiv 0 \mod 3$\\
$\frac{n+2}{3},$ & $n \equiv 1 \mod 3$\\
$\frac{n+1}{3},$ & $n \equiv 2 \mod 3$\\
\end{tabular}
\right.  .$$

\subsubsection{Traces of $T_{10}$, $T_{11}$, $T_{12}$ and $T_{13}$} $ $

$T_{10}$, $T_{11}$, $T_{12}$ and $T_{13}$ are conjugate to the diagonal matrix $T = diag(\xi_4, \xi_4, \xi_4^3, \xi_4^3)$. It suffices therefore to calculate the trace $Tr(T^{-1}, Sym^n(V))$. One has $$(3E_{11} + 3E_{22} + E_{33} + E_{44}). \Lambda_{a, b, c} = (n + 2b) \Lambda_{a, b, c}$$ and therefore $\rho_n(T^{-1})(\Lambda_{a, b, c}) = \xi_4^{n+2b} \Lambda_{a, b, c}$. Now, 

\begin{align}
Tr(T^{-1}, Sym^n(V)) &= \sum_{a=0}^n \sum_{b=0}^a \sum_{c=0}^b  \xi_4^{n+2b} = \xi_4^{n} \sum_{a=0}^n \sum_{b=0}^a (-1)^{b}(b+1). \nonumber
\end{align}

One can see that
$\sum_{b=0}^a (-1)^{b}(b+1) = 
\left\{ 
\begin{tabular}{rl}
$\frac{a+2}{2},$ & $a$ even\\
$-\frac{a+1}{2},$ & $a$ odd\\
\end{tabular}
\right.$. Therefore one has
$$Tr(T^{-1}, Sym^n (V)) = 
\left\{ 
\begin{tabular}{rl}
$\frac{n+2}{2},$ & $n \equiv 0 \mod 4$\\
$-\frac{n+2}{2},$ & $n \equiv 2 \mod 4$\\
$0$, & $n$ odd.
\end{tabular} \right. .$$

\subsubsection{Traces of $T_{14}$, $T_{15}$, $T_{16}$ and $T_{17}$} $ $

$T_{14}$, $T_{15}$ are conjugate to the diagonal matrix $T = diag(1, 1, \xi_4^3, \xi_4)$. As $T_{16} = -T_{14}$, $T_{17} = -T_{15}$, 
$$ Tr(T_{16}^{-1}, Sym^n(V)) = Tr(T_{17}^{-1}, Sym^n(V)) = (-1)^n Tr(T^{-1}, Sym^n(V)).$$

One has $$(E_{33} + 3E_{44}). \Lambda_{a, b, c} = (3n-2a-b) \Lambda_{a, b, c}$$ and $\rho_n(T^{-1})(\Lambda_{a, b, c}) = \xi_4^{3n-2a-b} \Lambda_{a, b, c}$. Therefore

\begin{align}
Tr(T^{-1}, Sym^n(V)) &= \xi_4^{3n} \sum_{a=0}^n (-1)^{a} \sum_{b=0}^a \xi_4^{-b}(b+1). \nonumber
\end{align}

Thus, one can see that
$$Tr(T^{-1}, Sym^n (V)) = 
\left\{ 
\begin{tabular}{rl}
$\frac{n+2}{2},$ & $n$ even \\
$\frac{n+3}{2},$ & $n \equiv 1 \mod 4$ \\
$\frac{n+1}{2},$ & $n \equiv 3 \mod 4$ \\
\end{tabular}
\right.  .$$

\subsubsection{Traces of $T_{18}$, $T_{19}$, $T_{20}$, $T_{21}$, $T_{43}$, $T_{44}$, $T_{45}$ and $T_{46}$} $ $

$T_{18}$, $T_{19}$, $T_{20}$, $T_{21}$ are conjugate to $T = diag(\xi_5, \xi_5^2, \xi_5^3, \xi_5^4)$, and $T_{43}$, $T_{44}$, $T_{45}$, $T_{46}$ are conjugate to -$T$.

One has $$(4E_{11} + 3E_{22} + 2E_{33} + E_{44}). \Lambda_{a, b, c} = (n+a+b+c) \Lambda_{a, b, c}$$ and therefore $\rho_n(T^{-1})(\Lambda_{a, b, c}) = \xi_5^{n+a+b+c} \Lambda_{a, b, c}$. Therefore

\begin{align}
Tr(T^{-1}, Sym^n(V)) &= \xi_5^{n} \sum_{a=0}^n \xi_5^{a} \sum_{b=0}^a \xi_5^b  \sum_{c=0}^b \xi_5^c. \nonumber
\end{align}

Hence one can finally see that
$$Tr(T^{-1}, Sym^n (V)) = 
\left\{ 
\begin{tabular}{rl}
$1,$ & $n \equiv 0 \mod 5$ \\
$-1,$ & $n \equiv 1 \mod 5$ \\
$0,$ & otherwise
\end{tabular}
\right.  .$$

\subsubsection{Traces of $T_{25}$, $T_{26}$, $T_{29}$ and $T_{30}$} $ $

$T_{25}$, $T_{26}$ are conjugate to the diagonal matrix $T = diag(1, 1, \xi_6^5, \xi_6)$ and $T_{29} = -T_{25}$, $T_{30} = -T_{26}$.

One has $$(E_{33} + 5E_{44}). \Lambda_{a, b, c} = (5n-4a-b) \Lambda_{a, b, c}$$ and $\rho_n(T^{-1})(\Lambda_{a, b, c}) = \xi_6^{5n-4a-b} \Lambda_{a, b, c}$. Therefore

\begin{align}
Tr(T^{-1}, Sym^n(V)) &= \xi_6^{5n} \sum_{a=0}^n \xi_6^{-4a} \sum_{b=0}^a \xi_6^{-b}(b+1) \nonumber
\end{align}

and we have
$$\sum_{b=0}^a \xi_6^{-b}(b+1) = 
\left\{ 
\begin{tabular}{rl}
$1 + \frac{a}{2}(1 - \xi_6^{4} - \xi_6^{5}),$ & $a \equiv 0 \mod 6$\\
$1 + 2 \xi_6^5 + \frac{a-1}{2}(1 - \xi_6^{4} + \xi_6^{5}),$ & $a \equiv 1 \mod 6$\\
$1 + 2 \xi_6^5 + 3 \xi_6^4 + \frac{a-2}{2}(1 + \xi_6^4 + \xi_6^5),$ & $a \equiv 2 \mod 6$\\
$2 \xi_6^5 + 3 \xi_6^4 + 3 \xi_6^3 + \frac{a-3}{2}(\xi_6^3 - \xi_6^2 - \xi_6),$ & $a \equiv 3 \mod 6$\\
$3 \xi_6^2 + 3 \xi_6^3 + 3 \xi_6^4 + \frac{a-4}{2}(\xi_6^3 + \xi_6^2 - \xi_6),$ & $a \equiv 4 \mod 6$\\
$\frac{a+1}{2}(\xi_6 + \xi_6^{2} + \xi_6^{3}),$ & $a \equiv 5 \mod 6$\\
\end{tabular}
\right..$$

By using this, one can show that 
$$Tr(T^{-1}, Sym^n (V)) = 
\left\{ 
\begin{tabular}{rl}
$n + 1,$ & $n \equiv 0 \mbox{ or } 5 \mod 6$\\
$n + 2,$ & $n \equiv 1 \mbox{ or } 4 \mod 6$\\
$n + 3,$ & $n \equiv 2 \mbox{ or } 3 \mod 6$\\
\end{tabular}
\right..$$

\subsubsection{Traces of $T_{31}$, $T_{32}$, $T_{33}$, $T_{34}$, $T_{35}$, $T_{36}$, $T_{37}$ and $T_{38}$} $ $

In this case, these torsion elements are conjugate to the diagonal matrix $T = diag(\xi_6^5, \xi_3^2, \xi_3, \xi_6)$ .

One has $$(E_{11} + 2E_{22} + 4E_{33} + 5E_{44}). \Lambda_{a, b, c} = (5n-a-2b-c) \Lambda_{a, b, c}$$ and $\rho_n(T^{-1})(\Lambda_{a, b, c}) = \xi_6^{5n-a-2b-c} \Lambda_{a, b, c}$. Therefore

\begin{align}
Tr(T^{-1}, Sym^n(V)) &= \xi_6^{5n} \sum_{a=0}^n \xi_6^{-a} \sum_{b=0}^a \xi_6^{-2b} \sum_{c=0}^b \xi_6^{-c}. \nonumber
\end{align}

And one can see that

$$Tr(T^{-1}, Sym^n (V)) = 
\left\{ 
\begin{tabular}{rl}
$1,$ & $a \equiv 0 \mod 6$\\
$-1,$ & $a \equiv 2 \mod 6$\\
$0,$ & otherwise\\
\end{tabular}
\right..$$

\subsubsection{Traces of $T_{39}$, $T_{40}$, $T_{41}$ and $T_{42}$} $ $

These torsion elements are conjugate to $T = diag(\xi_8, \xi_8^3, \xi_8^5, \xi_8^7)$. Therefore it suffices to calculate the trace $Tr(T^{-1}, Sym^n(V))$. One has $$(7E_{11} + 5E_{22} + 3E_{33} + E_{44}). \Lambda_{a, b, c} = (n + 2a + 2b + 2c) \Lambda_{a, b, c}$$ and therefore $\rho_n(T^{-1})(\Lambda_{a, b, c}) = \xi_8^{n+2a+2b+2c} \Lambda_{a, b, c}$. Now, 

\begin{align}
Tr(T^{-1}, Sym^n(V)) &= \xi_8^n \sum_{a=0}^n \xi_8^{2a} \sum_{b=0}^a \xi_8^{2b} \sum_{c=0}^b  \xi_8^{2c}. \nonumber
\end{align}

One can see that
$\sum_{b=0}^a \xi_8^{2b} \sum_{c=0}^b  \xi_8^{2c} = 
\left\{ 
\begin{tabular}{rl}
$1,$ & $a \equiv 0 \mod 4$\\
$\xi_4,$ & $a \equiv 1 \mod 4$\\
$0,$ & otherwise\\
\end{tabular}
\right.$.

and

$$\sum_{a=0}^n \xi_8^{2a} \sum_{b=0}^a \xi_8^{2b} \sum_{c=0}^b  \xi_8^{2c} = 
\left\{ 
\begin{tabular}{rl}
$1,$ & $a \equiv 0 \mod 4$\\
$0,$ & otherwise\\
\end{tabular}
\right. .$$

Therefore one has
$Tr(T^{-1}, Sym^n (V)) = 
\left\{ 
\begin{tabular}{rl}
$(-1)^{n/4},$ & $n \equiv 0 \mod 4$\\
$0,$ & otherwise\\
\end{tabular} \right. $.

\subsubsection{Traces of $T_{47}$ and $T_{48}$} $ $

These torsion elements are conjugate to $T = diag(\xi_{12}, \xi_{12}^5, \xi_{12}^7, \xi_{12}^{11})$. One has $$(11 E_{11} + 7E_{22} + 5E_{33} + E_{44}). \Lambda_{a, b, c} = (n + 4a + 2b + 4c) \Lambda_{a, b, c}$$ and therefore $\rho_n(T^{-1})(\Lambda_{a, b, c}) = \xi_{12}^{n+4a+2b+4c} \Lambda_{a, b, c}$. Now, 

\begin{align}
Tr(T^{-1}, Sym^n(V)) &= \xi_{12}^n \sum_{a=0}^n \xi_{3}^{a} \sum_{b=0}^a \xi_{6}^{b} \sum_{c=0}^b  \xi_{3}^{c}. \nonumber
\end{align}

One has,
$$Tr(T^{-1}, Sym^n (V)) = 
\left\{ 
\begin{tabular}{rl}
$1,$ & $n \equiv 0 \mbox{ or } 2 \mod 12$\\
$-1,$ & $n \equiv 6 \mbox{ or } 8 \mod 12$\\
$0,$ & otherwise\\
\end{tabular} \right. .$$

\subsubsection{Traces of $T_{49}$, $T_{50}$, $T_{51}$, $T_{52}$, $T_{53}$, $T_{54}$, $T_{55}$ and $T_{56}$} $ $

$T_{49}$, $T_{50}$, $T_{51}$, $T_{52}$ are conjugate to $T = diag(\xi_4, \xi_3, \xi_3^2, \xi_4^3)$. Hence, it suffices to calculate the trace $Tr(T^{-1}, Sym^n(V))$. One has $$(9E_{11} + 8E_{22} + 4E_{33} + 2E_{44}). \Lambda_{a, b, c} = (3n + a + 4b + c) \Lambda_{a, b, c}$$ and therefore $\rho_n(T^{-1})(\Lambda_{a, b, c}) = \xi_{12}^{3n + a + 4b + c} \Lambda_{a, b, c}$. Now, 
\begin{align}
Tr(T^{-1}, Sym^n(V)) &= \xi_{4}^{n} \sum_{a=0}^n \xi_{12}^{a} \sum_{b=0}^a \xi_{3}^{b} \sum_{c=0}^b  \xi_{12}^{c}. \nonumber
\end{align}
By using the fact that $1 + \xi_{12}^4 + \xi_{12}^8=0$ one has
$$\xi_{12}^{a} \sum_{b=0}^a \xi_{3}^{b} \sum_{c=0}^b  \xi_{12}^{c} = 
\left\{ 
\begin{tabular}{rl}
$1,$ & $a \equiv 0 \mod 12$\\
$\xi_{12} + \xi_{12}^5 + \xi_{12}^6 = i-1,$ & $a \equiv 1 \mod 12$\\
$\xi_{12}^2 + \xi_{12}^7 + \xi_{12}^{10} + \xi_{12}^{11} = 1 - i,$ & $a \equiv 2 \mod 12$\\
$2 \xi_{12}^3 + \xi_{12}^4 + \xi_{12}^8 = 2i - 1,$ & $a \equiv 3 \mod 12$\\
$\xi_{12}^4 + \xi_{12}^8 + 2 \xi_{12}^9 + 1 = -2i,$ & $a \equiv 4 \mod 12$\\
$2 \xi_{12} + 2 \xi_{12}^5 = 2i,$ & $a \equiv 5 \mod 12$\\
$\xi_{12}^6 + 2 \xi_{12}^7 + 2 \xi_{12}^{11}   = - 2i - 1,$ & $a \equiv 6 \mod 12$\\
$1 + 2\xi_{12}^3 + \xi_{12}^{7} + \xi_{12}^{11} = 1 + i,$ & $a \equiv 7 \mod 12$\\
$\xi_{12}^4 + \xi_{12}^8 + \xi_{12}^9 = - i - 1,$ & $a \equiv 8 \mod 12$\\
$2 \xi_{12}^2 + \xi_{12}^6 + 2 \xi_{12}^{10} = 1,$ & $a \equiv 9 \mod 12$\\
$2 \xi_{12}^3 + 2 \xi_{12}^7 + 2 \xi_{12}^{11} = 0,$ & $a \equiv 10 \mod 12$\\
$2 \xi_{12}^3 + 2 \xi_{12}^7 + 2 \xi_{12}^{11} = 0,$ & $a \equiv 11 \mod 12$\\
\end{tabular}
\right.$$

and

$$Tr(T^{-1}, Sym^n (V)) = 
\left\{ 
\begin{tabular}{rl}
$1,$ & $a \equiv 0 \mbox{ or } 6 \mbox{ or } 7 \mod 12$\\
$-1,$ & $a \equiv 1 \mbox{ or } 2 \mbox{ or } 8 \mod 12$\\
$2,$ & $a \equiv 3 \mod 12$\\
$-2,$ & $a \equiv 5 \mod 12$\\
$0,$ & otherwise \\
\end{tabular}
\right..$$

\subsection{Traces with respect to highest weight representations $\m_\lambda$}

By using Theorem~\ref{generalize}, we give, in the following theorem, a complete description of the traces of torsion elements with respect to every finite dimensional irreducible representation of $\Sp$.

We introduce the following matrices that will be helpful in the aforementioned description.

\[
M_{B, C} = \left( \begin{array}{cc}
\frac{(n_1+2)(n_2+1)}{2} & 0 \\
0 & -\frac{(n_1+2)(n_2+1)}{2} \\
 \end{array}  \right),  
M_{D, E} = \left( \begin{array}{cccc}
\frac{n_1+n_2+3}{3} & \frac{n_2-n_1-1}{3} & -\frac{2n_2+2}{3} \\
-\frac{2n_1+4}{3} & \frac{2n_1+4}{3} & 0\\
\frac{n_1-n_2 + 1}{3} & -\frac{n_1 + n_2 + 3}{3} & \frac{2n_2 + 2}{3} \\ \end{array}  \right), 
\]
\[
M_{F} = \left( \begin{array}{cccc}
\frac{n_1 + n_2 + 3}{3} & -\frac{n_1 - n_2 + 1}{3} & \frac{n_2 + 1}{3} \\
\frac{n_1 + 2}{3} & -\frac{n_1 + 2}{3} & 0 \\
\frac{n_1 - n_2 + 1}{3} & -\frac{n_1 +n_2 +3}{3} & -\frac{n_2 +1}{3} \end{array}  \right),
M_{J} = \left( \begin{array}{cccc}
\frac{n_1+2}{2} & 0 & -\frac{n_1+2}{2} & 0\\
\frac{n_1+n_2+3}{2} & \frac{n_2+1}{2} & -\frac{n_1-n_2+1}{2} & \frac{n_2+1}{2}\\
\frac{n_1+2}{2} & 0 & -\frac{n_1+2}{2} & 0 \\
\frac{n_1-n_2+1}{2} & -\frac{n_2+1}{2} & -\frac{n_1+n_2+3}{2} & -\frac{n_2+1}{2}\\ \end{array}  \right), 
\]
\[
M_{F_2} = \left( \begin{array}{cccccc}
n_1 - n_2 + 1 & n_1 - n_2 + 1 & -n_2 -1 & -n_1 - n_2 -3 & -n_1 -n_2 -3 & -n_2 -1 \\
n_1 + 2 & n_1 + 2 & 0 & -n_1 -2 & -n_1 -2 & 0 \\
n_1 + n_2 + 3 & n_1 + n_2 + 3 & n_2 + 1 & -n_1 + n_2 -1 & -n_1 + n_2 -1 & n_2 + 1 \\
n_1 + n_2 + 3 & n_1 + n_2 + 3 & n_2 + 1 & -n_1 + n_2 -1 & -n_1 + n_2 -1 & n_2 + 1 \\
n_1 + 2 & n_1 + 2 & 0 & -n_1 -2 & -n_1 -2 & 0 \\
n_1 -n_2 + 1 & n_1 -n_2 + 1 & -n_2 -1 & -n_1 -n_2 -3 & -n_1 -n_2 -3 & -n_2 -1 \end{array}  \right),
\]
\[
M_{G, H, I} =  \left( \begin{array}{cccc}
\frac{n_1 + 2}{2} & 0 & -\frac{n_1 + 2}{2} & 0 \\
0 & \frac{n_2 + 1}{2} & 0 & -\frac{n_2 + 1}{2} \\
-\frac{n_1 + 2}{2} & 0 & \frac{n_1 + 2}{2} & 0 \\
0 & -\frac{n_2 + 1}{2} & 0 & \frac{n_2 + 1}{2} \end{array}  \right),
\]
\[
M_{K} = \left( \begin{array}{rrrrrrrrrrrrrrrrrrrr}
1 & 0 & 0 & -1 & 0 \\
-1 & 0 & 0 & 1 & 0 \\
0 & 1 & -1 & 0 & 0 \\
0 & 0 & 0 & 0 & 0 \\
0 & -1 & 1 & 0 & 0 \end{array}  \right),
M_{L, M} = \left( \begin{array}{rrrrrrrrrrrr}
1 & 0 & 0 & 0 & -1 & 0\\
0 & 0 & 0 & 0 & 0 & 0\\
-1 & 0 & 0 & 0 & 1 & 0\\
0 & 1 & 0 & -1 & 0 & 0\\
0 & 0 & 0 & 0 & 0 & 0\\
0 & -1 & 0 & 1 & 0 & 0 \end{array}  \right), 
\]
\[
M_{N} = \left( \begin{array}{rrrrrrrrrrrrrrrrrrr}
1 & 0 & 1 & 0 & -1 & 0 & -1 & 0 \\
0 & -1 & 0 & 0 & 0 & 1 & 0 & 0 \\
0 & 0 & 0 & 0 & 0 & 0 & 0 & 0 \\
0 & 1 & 0 & 0 & 0 & -1 & 0 & 0 \\
-1 & 0 & -1 & 0 & 1 & 0 & 1 & 0 \\
0 & 1 & 0 & 0 & 0 & -1 & 0 & 0 \\
0 & 0 & 0 & 0 & 0 & 0 & 0 & 0 \\
0 & -1 & 0 & 0 & 0 & 1 & 0 & 0 \end{array}  \right),
\]
\[
M_M = \left( \begin{array}{rrrrrrrrrrrrrrrrrrrrrrr}
1 & 0 & 2 & 0 & 1 & 0 & -1 & 0 & -2 & 0 & -1 & 0 \\
0 & -2 & 0 & -2 & 0 & 0 & 0 & 2 & 0 & 2 & 0 & 0 \\
1 & 0 & 2 & 0 & 1 & 0 & -1 & 0 & -2 & 0 & -1 & 0 \\
0 & -1 & 0 & -1 & 0 & 0 & 0 & 1 & 0 & 1 & 0 & 0 \\
0 & 0 & 0 & 0 & 0 & 0 & 0 & 0 & 0 & 0 & 0 & 0 \\
0 & 1 & 0 & 1 & 0 & 0 & 0 & -1 & 0 & -1 & 0 & 0 \\
-1 & 0 & -2 & 0 & -1 & 0 & 1 & 0 & 2 & 0 & 1 & 0 \\
0 & 2 & 0 & 2 & 0 & 0 & 0 & -2 & 0 & -2 & 0 & 0 \\
-1 & 0 & -2 & 0 & -1 & 0 & 1 & 0 & 2 & 0 & 1 & 0 \\
0 & 1 & 0 & 1 & 0 & 0 & 0 & -1 & 0 & -1 & 0 & 0 \\
0 & 0 & 0 & 0 & 0 & 0 & 0 & 0 & 0 & 0 & 0 & 0 \\
0 & -1 & 0 & -1 & 0 & 0 & 0 & 1 & 0 & 1 & 0 & 0 \end{array}  \right),
\] 
\[
M_O = \left( \begin{array}{rrrrrrrrrrrrrrrrrrrrrrrr}
1 & 0 & -1 & 0 & 1 & 0 & -1 & 0 & 1 & 0 & -1 & 0\\
-1 & 1 & 0 & -1 & 1 & 0 & -1 & 1 & 0 & -1 & 1 & 0\\
-1 & 0 & 1 & 0 & -1 & 0 & 1 & 0 & -1 & 0 & 1 & 0\\
2 & -1 & -1 & 1 & 0 & 0 & 0 & -1 & 1 & 1 & -2 & 0\\
0 & 0 & 0 & 0 & 0 & 0 & 0 & 0 & 0 & 0 & 0 & 0\\
-2 & 1 & 1 & -1 & 0 & 0 & 0 & 1 & -1 & -1 & 2 & 0\\
1 & 0 & -1 & 0 & 1 & 0 & -1 & 0 & 1 & 0 & -1 & 0\\
1 & -1 & 0 & 1 & -1 & 0 & 1 & -1 & 0 & 1 & -1 & 0\\
-1 & 0 & 1 & 0 & -1 & 0 & 1 & 0 & -1 & 0 & 1 & 0\\
0 & 1 & -1 & -1 & 2 & 0 & -2 & 1 & 1 & -1 & 0 & 0\\
0 & 0 & 0 & 0 & 0 & 0 & 0 & 0 & 0 & 0 & 0 & 0\\
0 & -1 & 1 & 1 & -2 & 0 & 2 & -1 & -1 & 1 & 0 & 0 \end{array}  \right).
\]

\begin{thm}
\label{traces} The following is a complete list of the traces of each representative of conjugacy classes of torsion elements in $\mathrm{Sp}_4(\mathbb{Z})$ with respect to every finite dimensional irreducible representation $\mathcal{M}_\lambda$ of $\mathrm{Sp}_4$ where $\lambda = m_1 \lambda_1 + m_2 \lambda_2 = (m_1+m_2)e_1 + m_2e_2$. We denote $n_1 = m_1 + m_2$ and $n_2 = m_2$. In the table we use the following notation: if $M$ is an $n \times n$ square matrix, we denote by $M(n_1, n_2)$ the entry of the matrix $M$ in the row $n_1 = m_1 + m_2$ mod $n$ and column $n_2=m_2$ mod $n$ (where we are numbering the rows and columns from $0$ to $n-1$).

{ \begin{center} 
\scriptsize\renewcommand{\arraystretch}{2}
\begin{longtable}{|c|c|c|c|c|c|c|}
\hline
Case &Torsion Element $T$ &  $Tr(T^{-1}, \m_\lambda)$ \\
\hline
\hline
A & $T_1$ &  $\frac{1}{6}(n_1+2)(n_2+1)((n_1+2)^2 - (n_2+1)^2)$ \\
\hline
A & $T_2$ &  $(-1)^{m_1}\frac{1}{6}(n_1+2)(n_2+1)((n_1+2)^2 - (n_2+1)^2)$ \\
\hline
B, C &$T_3$, $T_4$ & $M_{B, C}(n_1, n_2)$ \\
\hline
D, E &$T_5, T_6, T_{7}$ &  $M_{D, E}(n_1, n_2)$ \\
\hline
D, E &$T_{22}, T_{23}, T_{24}$ &  $(-1)^{m_1} M_{D, E}(n_1,  n_2)$ \\
\hline
F &$T_8, T_9$ & $M_{F}(n_1,  n_2)$  \\
\hline
F &$T_{27}, T_{28}$ & $(-1)^{m_1} M_{F}(n_1,  n_2)$  \\
\hline
F &$T_{25},T_{26} $ & $M_{F_2}(n_1,  n_2)$ \\
\hline
F &$T_{29},T_{30} $ & $(-1)^{m_1} M_{F_2}(n_1,  n_2)$ \\
\hline
G, H, I &$T_{10}, T_{11}, T_{12}, T_{13}$ & $M_{G, H, I}(n_1,  n_2)$  \\
\hline
J &$T_{14},  T_{15}$ & $M_{J}(n_1,  n_2)$ \\
\hline
J &$T_{16}, T_{17}$ &  $(-1)^{m_1} M_{J}(n_1,  n_2)$ \\
\hline
K &$T_{18}, T_{19}, T_{20}, T_{21}$ & $M_{K}(n_1,  n_2)$\\
\hline
K &$T_{43}, T_{44}, T_{45}, T_{46}$ & $(-1)^{m_1} M_{K}(n_1,  n_2)$\\
\hline
L, M &$T_{31}, T_{32}, T_{33}, T_{34}, T_{35}, T_{36}, T_{37}, T_{38}$ & $M_{L, M}(n_1, n_2)$ \\
\hline
N &$T_{39}, T_{40}, T_{41}, T_{42}$ & $M_{N}(n_1, n_2)$ \\
\hline
M & $T_{47},T_{48}$ & $M_M(n_1,  n_2)$ \\
\hline
O & $T_{49}, T_{50}, T_{51}, T_{52}$ & $M_O(n_1, n_2)$\\
\hline
O & $T_{53},T_{54},T_{55},T_{56}$ & $(-1)^{m_1} M_O(n_1, n_2)$\\
\hline
\hline
\hline
\caption{Traces of torsion elements for $\lambda = m_1 \lambda_1 + m_2\lambda_2$, $n_1=m_1+m_2$ and $n_2 = m_2$}\label{tracetorsionhighestweight}
\end{longtable}
\end{center}
}
\end{thm}


\section{Homological Euler Characteristic}\label{homeuler}

\subsection{Homological Euler Characteristic with respect to $Sym^n(V)$}

With the calculations made so far one can finally obtain the homological Euler characteristic with respect to every irreducible finite dimensional representation of $\mathrm{Sp}_4$ by using ~\eqr{hecT}, and in particular with respect to the symmetric power representations $Sym^n(V)$. When $n$ is odd, it is clear that this Euler characteristic is $0$. In the next theorem we give an explicit formula of the Euler characteristic for $n$ even. 

\begin{thm}\label{Hom Symm}
For an even integer $n=2k$, one has the following description of the homological Euler characteristic $$\chi_h(\mathrm{Sp}_4(\mathbb{Z}), Sym^{n} V) = -\frac{1}{4320} n^3 - \frac{1}{720} n^2 + a_k n + b_k + c_k$$ where
$$c_k = 
\left\{ 
\begin{tabular}{rl}
$\frac{(-1)^{k/2}}{2},$ & $k$ even\\
$0,$ & otherwise\\
\end{tabular} \right. .$$
{ \begin{center}
and $a_k$, $b_k$ are determined by the congruence of $k$ modulo $30$, and are given by
\scriptsize\renewcommand{\arraystretch}{2}
\begin{longtable}{|c|c|c|c|c|c|c|c|c|c|c|c|c|c|c|c|c|c|c|c|c|c|c|c|c|c|c|c|c|c|c|} 
\hline
$k$ mod $30$ & 0& 1& 2& 3& 4& 5& 6& 7& 8& 9& 10& 11& 12& 13& 14 \\
\hline
$a_i$ & $-\frac{2}{15}$ & $-\frac{11}{120}$ & $-\frac{7}{90}$ & $-\frac{11}{120}$ & $-\frac{2}{15}$ & $-\frac{13}{360}$ & $-\frac{2}{15}$ & $-\frac{11}{120}$ & $-\frac{7}{90}$ & $-\frac{11}{120}$ & $-\frac{2}{15}$ & $-\frac{13}{360}$ & $-\frac{2}{15}$ & $-\frac{11}{120}$ & $-\frac{7}{90}$  \\
\hline
$b_i$ & $\frac{3}{2}$ & $-\frac{437}{540}$ & $-\frac{41}{270}$ & $-\frac{7}{20}$ & $-\frac{331}{270}$ & $\frac{79}{108}$ & $\frac{7}{10}$ & $-\frac{437}{540}$ & $-\frac{257}{270}$ & $\frac{9}{20}$ & $-\frac{23}{54}$ & $-\frac{37}{540}$ & $\frac{7}{10}$ & $-\frac{869}{540}$ & $-\frac{41}{270}$ 
\\
\hline
\hline
$k$ mod $30$ & 15& 16& 17& 18& 19& 20& 21& 22& 23& 24& 25& 26& 27& 28& 29 \\
\hline
$a_i$ & $-\frac{11}{120}$ & $-\frac{2}{15}$ & $-\frac{13}{360}$ & $-\frac{2}{15}$ & $-\frac{11}{120}$ & $-\frac{7}{90}$ & $-\frac{11}{120}$ & $-\frac{2}{15}$ & $-\frac{13}{360}$ & $-\frac{2}{15}$ & $-\frac{11}{120}$ & $-\frac{7}{90}$ & $-\frac{11}{120}$ & $-\frac{2}{15}$ & $-\frac{13}{360}$ \\
\hline
$b_i$ & $\frac{5}{4}$ & $-\frac{331}{270}$ & $-\frac{37}{540}$ & $-\frac{1}{10}$ & $-\frac{437}{540}$ & $\frac{35}{54}$ & $\frac{9}{20}$ & $-\frac{331}{270}$ & $-\frac{469}{540}$ & $\frac{7}{10}$ & $-\frac{1}{108}$ & $-\frac{41}{270}$ & $\frac{9}{20}$ & $-\frac{547}{270}$ & $-\frac{37}{540}$ 
\\
\hline
\caption{Coefficients for the homological Euler characteristics with respect to $Sym^n(V)$}\label{symeuler}
\end{longtable}
\end{center}
}
\end{thm}

\begin{proof}
For the proof of this theorem, one has to use~\eqr{hecT}, the orbifold Euler characteristics determined in Section~\ref{Euler char} and the traces calculated in the Subsection~\ref{Traces Sym}. The orbifold Euler characteristics are rational constants and the traces with respect to the symmetric power representations are polynomials in $n = 2k$. Taking into account the labels of torsion elements in Table ~\ref{orbifoldeuler}, the only elements whose traces are polynomials of degree greater than $1$ are those of type $A$, they will therefore determine the coefficients of $n^3$ and $n^2$ in the formula for the Euler characteristics of the symmetric power representations. The term $c_k$ in the formula is the contribution of the elements of type $N$, the sum of the contribution of the rest of the torsion elements will only depend on $n$ modulo $60$ (or $k$ module $30$) and determines the coefficients $a_k$ and $b_k$ of the formula.
\end{proof}

Therefore we can see that the degree $3$ polynomial determining the homological Euler charactersitic of the $2k$-th symmetric power depends only on $k$ modulo $60$.

In the following table we make use of the results obtained to give the homological Euler characteristics for the first $150$ even symmetric powers. If one numerates the lines from $0$ to $4$ and the columns from $0$ to $15$, then the value of $\chi_h(\mathrm{Sp}_4(\mathbb{Z}), Sym^{2k} V)$ can be found in line $i$ and column $j$ where $i$ and $j$ are such that $k = 15*i + j$. \\

\begin{center}
\scriptsize\renewcommand{\arraystretch}{2}
\begin{longtable}{|c||c|c|c|c|c|c|c|c|c|c|c|c|c|c|c|c|c|c|c|c|c|c|c|c|c|c|c|c|c|c|c|} 
\hline
$i \backslash j$ & 0& 1& 2& 3& 4& 5& 6& 7& 8& 9& 10& 11& 12& 13& 14 \\
\hline
\hline
0 & 2& -1& -1& -1& -2& 0& -2& -3& -3& -3& -6& -4& -6& -9& -9\\
\hline
1& -9& -14& -12& -18& -19& -19& -23& -30& -28& -34& -37& -41& -45& -54& -52\\
\hline
2& -62& -67& -71& -79& -90& -88& -102& -109& -117& -125& -138& -140& -158& -167& -175\\
\hline
3& -187& -206& -208& -230& -241& -253& -269& -290& -296& -322& -335& -351& -371& -398& -404\\
\hline
4& -434& -453& -473& -497& -526& -536& -574& -595& -619& -647& -682& -696& -738& -765& -793\\
\hline
\caption{Homological Euler Characteristic $\chi_h(\Sp(\Z), Sym^{2k} V)$, for $k=15i + j$ }\label{homologicalsymeuler}
\end{longtable}
\end{center}

\subsection{Homological Euler characteristic for general irreducible finite dimensional representation}

With the calculations we already made, we have all the traces of torsion elements and the orbifold Euler characteristics of their centralizers. Therefore one can already calculate the homological Euler characteristic. One can easily see in particular that for $m_1$ odd, one has $\chi_h(\Gamma, \m_\lambda) = 0$, so the interesting examples come from the cases $m_1$ even. 

By using the formula (\ref{eq:hecT}) for homological Euler characteristics 
\begin{equation}
    \chi_h(\Gamma, \m_\lambda) = \sum_{(T)} \chi(C_\Gamma(T)) Tr(T^{-1}, \m_\lambda), \nonumber
\end{equation}
one can group the terms in this sum taking into account the Table \ref{tracetorsionhighestweight} and check the coefficients multiplying each matrix by taking the sum of the orbifold Euler characteristics $\chi(C_\Gamma(T))$ for all the torsion elements associated to the given matrix in the table.
Therefore, if $m_1$ is even, one has:
\begin{align}\label{homEuler}
\chi_h(\Gamma,V) &= \chi_{A} M_A(n_1,  n_2) + \chi_{B, C} M_{B, C}(n_1,  n_2) + \chi_{D, E} M_{D, E}(n_1,  n_2) + \chi_{F} M_F(n_1,  n_2)\nonumber \\
& + \chi_{F_2} M_{F_2}(n_1,  n_2) + \chi_{G,H,I} M_{G,H,I}(n_1,  n_2) + \chi_{J} M_{J}(n_1,  n_2) + \chi_{K} M_{K}(n_1,  n_2) \nonumber \\
&+ \chi_{L, M} M_{L, M}(n_1,  n_2) + \chi_{M} M_{M}(n_1,  n_2) + \chi_{N} M_{N}(n_1,  n_2) + \chi_{0} M_{O}(n_1,  n_2) \nonumber
\end{align}

where $n_1 = m_1 + m_2$, $n_2 = m_2$ and $M_A(n_1,  n_2)$ denotes $\frac{1}{6}(n_1+2)(n_2+1)((n_1+2)^2-(n_2+1)^2)$ (and if $M$ is an $n \times n$ square matrix, we denote by $M(n_1,  n_2)$ the entry of the matrix $M$ in the row $n_1 = (m_1 + m_2)$ mod $n$ and column $n_2 = m_2$ mod $n$, where we are numbering the rows and columns from $0$ to $n-1$). On the other hand, for the coefficients one has
\begin{align}
\chi_A &= \chi(C(T_1)) + \chi(C(T_2)) = -\frac{1}{720}, \quad &\chi_J = \sum_{k=14}^{17} \chi(C(T_k)) = -\frac{1}{12}, \nonumber \\
\chi_{B, C} &= \chi(C(T_3)) + \chi(C(T_4)) = \frac{7}{144}, \quad &\chi_K = \sum_{k=18}^{21} \chi(C(T_k)) + \sum_{k=43}^{46} \chi(C(T_k))= \frac{4}{5}, \nonumber \\
\chi_{D, E} &= \sum_{k=5}^{7} \chi(C(T_k)) + \sum_{k=22}^{24} \chi(C(T_k)) = -\frac{1}{18}, \quad &\chi_{L, M} = \sum_{k=31}^{38} \chi(C(T_k)) = \frac{4}{9},\nonumber \\
\chi_{F} &= \sum_{k=8}^{9} \chi(C(T_k)) + \sum_{k=27}^{28} \chi(C(T_k)) = -\frac{1}{18}, \quad &\chi_{M} = \sum_{k=47}^{48} \chi(C(T_k)) = \frac{1}{6},\nonumber \\
\chi_{F_2} &= \sum_{k=25}^{26} \chi(C(T_k)) + \sum_{k=29}^{30} \chi(C(T_k)) = -\frac{1}{18}, \quad &\chi_{N} = \sum_{k=39}^{42} \chi(C(T_k)) = \frac{1}{2},\nonumber \\
\chi_{G,H,I} &= \sum_{k=10}^{13} \chi(C(T_k)) = -\frac{1}{24}, \quad &\chi_{O} = \sum_{k=49}^{56} \chi(C(T_k)) = \frac{1}{3}. \nonumber
\end{align}

We can make this expression even shorter by writting the term 
\begin{align}
&\chi_{D, E} M_{D, E}(n_1,  n_2) + \chi_{F} M_F(n_1,  n_2) + \chi_{F_2} M_{F_2}(n_1,  n_2) + \chi_{G,H,I} M_{G,H,I}(n_1,  n_2) + \chi_{J} M_{J}(n_1,  n_2) \nonumber \\
&+ \chi_{L, M} M_{L, M}(n_1,  n_2) + \chi_{M} M_{M}(n_1,  n_2) + \chi_{O} M_{O}(n_1,  n_2) \nonumber
\end{align}
in a $12$ by $6$ matrix $E$ given by:

\[
\left( 
\scriptsize\renewcommand{\arraystretch}{1.9}
\scalemath{1}{
\begin{array}{cccccccccccc}
-\frac{67n_1}{432}+\frac{n_2}{54}+\frac{47}{72} & -\frac{49n_1}{432}-\frac{n_2}{54}-\frac{185}{216} & -\frac{43n_1}{432}-\frac{43}{216} & -\frac{49n_1}{432}+\frac{n_2}{54}+\frac{29}{72} & -\frac{67n_1}{432}-\frac{n_2}{54}-\frac{275}{216} & -\frac{25n_1}{432}-\frac{25}{216}\\

-\frac{2n_1}{27}-\frac{n_2}{16}-\frac{91}{432} & -\frac{n_1}{54}-\frac{13n_2}{432}-\frac{53}{432} & -\frac{n_1}{54}+\frac{13n_2}{432}+\frac{7}{144} & -\frac{2n_1}{27}+\frac{n_2}{16}-\frac{37}{432} & -\frac{n_1}{54}-\frac{67n_2}{432}+\frac{325}{432} & -\frac{n_1}{54}+\frac{67n_2}{432}-\frac{119}{144}\\

\frac{n_1}{48}-\frac{2n_2}{27}+\frac{137}{216} & \frac{n_1}{16}+\frac{1}{8} & \frac{n_1}{48}+\frac{2n_2}{27}-\frac{119}{216} & \frac{n_1}{16}-\frac{2n_2}{27}+\frac{11}{216} & \frac{n_1}{48}+\frac{1}{24} & \frac{n_1}{16}+\frac{2n_2}{27}+\frac{43}{216}\\

\frac{n_1}{54}-\frac{31n_2}{432}-\frac{5}{16} & \frac{n_1}{54}+\frac{31n_2}{432}+\frac{167}{432} & \frac{2n_1}{27}+\frac{n_2}{48}+\frac{361}{432} & \frac{n_1}{54}-\frac{49n_2}{432}-\frac{11}{16} & \frac{n_1}{54}+\frac{49n_2}{432}+\frac{329}{432} & \frac{2n_1}{27}-\frac{n_2}{48}-\frac{233}{432}\\

-\frac{11n_1}{432}-\frac{11}{216} & \frac{31n_1}{432}+\frac{n_2}{54}-\frac{25}{216} & \frac{13n_1}{432}-\frac{n_2}{54}-\frac{1}{72} & \frac{7n_1}{432}+\frac{7}{216} & \frac{13n_1}{432}+\frac{n_2}{54}+\frac{29}{216} & \frac{31n_1}{432}-\frac{n_2}{54}+\frac{29}{72}\\

-\frac{11n_2}{432}-\frac{11}{432} & \frac{n_2}{16}+\frac{1}{16} & -\frac{43n_2}{432}-\frac{43}{432} & \frac{43n_2}{432}+\frac{43}{432} & -\frac{n_2}{16}-\frac{1}{16} & \frac{11n_2}{432}+\frac{11}{432}\\

-\frac{31n_1}{432}+\frac{n_2}{54}+\frac{11}{72} & -\frac{13n_1}{432}-\frac{n_2}{54}-\frac{5}{216} & -\frac{7n_1}{432}-\frac{7}{216} & -\frac{13n_1}{432}+\frac{n_2}{54}-\frac{7}{72} & -\frac{31n_1}{432}-\frac{n_2}{54}-\frac{95}{216} & \frac{11n_1}{432}+\frac{11}{216}\\

-\frac{2n_1}{27}+\frac{n_2}{48}-\frac{343}{432} & -\frac{n_1}{54}-\frac{49n_2}{432}+\frac{199}{432} & -\frac{n_1}{54}+\frac{49n_2}{432}-\frac{77}{144} & -\frac{2n_1}{27}-\frac{n_2}{48}+\frac{215}{432} & -\frac{n_1}{54}-\frac{31n_2}{432}+\frac{73}{432} & -\frac{n_1}{54}+\frac{31n_2}{432}-\frac{35}{144}\\

-\frac{n_1}{16}-\frac{2n_2}{27}-\frac{43}{216} & -\frac{n_1}{48}-\frac{1}{24} & -\frac{n_1}{16}+\frac{2n_2}{27}-\frac{11}{216} & -\frac{n_1}{48}-\frac{2n_2}{27}-\frac{169}{216} & -\frac{n_1}{16}-\frac{1}{8} & -\frac{n_1}{48}+\frac{2n_2}{27}+\frac{151}{216}\\

\frac{n_1}{54}-\frac{67n_2}{432}-\frac{17}{16} & \frac{n_1}{54}+\frac{67n_2}{432}+\frac{491}{432} & \frac{2n_1}{27}-\frac{n_2}{16}+\frac{37}{432} & \frac{n_1}{54}-\frac{13n_2}{432}+\frac{1}{16} & \frac{n_1}{54}+\frac{13n_2}{432}+\frac{5}{432} & \frac{2n_1}{27}+\frac{n_2}{16}+\frac{91}{432}\\

\frac{25n_1}{432}+\frac{25}{216} & \frac{67n_1}{432}+\frac{n_2}{54}-\frac{133}{216} & \frac{49n_1}{432}-\frac{n_2}{54}+\frac{59}{72} & \frac{43n_1}{432}+\frac{43}{216} & \frac{49n_1}{432}+\frac{n_2}{54}-\frac{79}{216} & \frac{67n_1}{432}-\frac{n_2}{54}+\frac{89}{72}\\

\frac{25n_2}{432}+\frac{25}{432} & -\frac{n_2}{48}-\frac{1}{48} & -\frac{7n_2}{432}-\frac{7}{432} & \frac{7n_2}{432}+\frac{7}{432} & \frac{n_2}{48}+\frac{1}{48} & -\frac{25n_2}{432}-\frac{25}{432}\\
 \end{array}  
 }
 \right) 
\]

We also introduce the notation 
\begin{align}
\tilde{\chi}_A(n_1,  n_2) &= \chi_A M_A(n_1,  n_2), &\tilde{\chi}_{B, C}(n_1,  n_2) &= \chi_{B, C} M_{B, C}(n_1,  n_2), \nonumber \\
\tilde{\chi}_K(n_1,  n_2) &= \chi_K M_K(n_1,  n_2), &\tilde{\chi}_N(n_1,  n_2) &= \chi_N M_N(n_1,  n_2). \nonumber 
\end{align}
We finally obtain the following result:

\begin{thm} \label{Homological formula}
We use the following notation: $E(m_2, \frac{m_1}{2})$ is the entry of $E$ in the row $m_2$ mod $12$ and column $\frac{m_1}{2}$ mod $6$, where we are numbering the rows from $0$ to $11$ and the columns from $0$ to $5$. Using the above notation, we can express the Euler characteristic of $\Sp(\Z)$ with coefficients in the highest weight representation $\m_\lambda$ with $\lambda=m_1\lambda_1 + m_2\lambda_2$ as
\[\chi_h(\Sp(\Z),\m_\lambda)=
\tilde{\chi}_A(n_1,  n_2) + \tilde{\chi}_{B,C}(n_1,  n_2) + \tilde{\chi}_K(n_1,  n_2) + \tilde{\chi}_N(n_1,  n_2) + E(m_2, \frac{m_1}{2}),
\]
when $m_1$ is even (and as we mentioned before $\chi_h(\Sp(\Z),\m_\lambda)=0$ otherwise). In this expression $n_1$ is, as usual, $m_1 + m_2$ and $n_2 = m_2$.
\end{thm}

To illustrate the simplicity of this formula we will calculate the homological Euler characteristic for the highest weight $(m_1, m_2) = (20,19)$. In this case one has $n_1 = m_1 + m_2 = 39$ and $n_2 = m_2 = 19$.
\begin{itemize}
\item $m_2 = 19 \equiv 7$ mod $12$ and $\frac{m_1}{2} = 10 \equiv 4$ mod $6$, therefore $E(m_2, \frac{m_1}{2}) = -\frac{n_1}{54} - \frac{31n_2}{432}+\frac{73}{432}$.
\item $\tilde{\chi}_A(n_1,  n_2) = -\frac{1}{720}\frac{1}{6}(n_1+2)(n_2+1)((n_1+2)^2-(n_2+1)^2)$.
\item $n_1 \equiv n_2 \equiv 1$ mod $2$, therefore $\tilde{\chi}_{B, C}(n_1,  n_2) = -\frac{7}{144} \frac{(n_1+2)(n_2+1)}{2}$.
\item $n_1 \equiv n_2 \equiv 4$ mod $5$, therefore $\tilde{\chi}_K(n_1,  n_2) = 0$.
\item $n_1 \equiv 7$ mod $8$ and $n_2 \equiv 3$ mod $8$, therefore $\tilde{\chi}_N(n_1,  n_2) = 0$.
\end{itemize}
And finally, for $\lambda = 20\lambda_1 + 19\lambda_2$ one has
\[
\chi_h(\Sp(\Z),\m_\lambda) = -\frac{39}{54}-\frac{31 \cdot 19}{432}+\frac{73}{432} -\frac{1}{720}\frac{1}{6}(41\cdot 20)(41^2-20^2)-\frac{7}{144} \frac{41\cdot 20}{2} = -265.
\]

The following table gives the homological Euler characteristics with respect to the representations $\m_\lambda$ for $m_1$ even, $m_1 < 15$ and $m_2 < 15$, but the reader could use Theorem \ref{Homological formula} to calculate the homological Euler characteristic for any other highest weight.

\begin{center}
\scriptsize\renewcommand{\arraystretch}{2}
\begin{longtable}{|c||c|c|c|c|c|c|c|c|c|c|c|c|c|c|c|c|c|c|c|c|c|c|c|c|c|c|c|c|c|c|c|} 
\hline
$m_1 \backslash m_2$ & 0& 1& 2& 3& 4& 5& 6& 7& 8& 9& 10& 11& 12& 13& 14 \\
\hline
\hline
0 & 2 & -1 & 0 & -1 & 1 & -1 & 1 & -4 & 1 & -6 & 4 & -4 & 2 & -9 & 4\\
\hline
2 & -1 & 0 & 1 & -1 & 1 & -1 & 0 & -3 & 0 & -2 & 2 & -8 & -1 & -12 & 1\\
\hline
4 & -1 & 0 & 0 & 0 & 0 & -3 & 0 & -5 & -1 & -7 & 1 & -12 & -5 & -16 & -4\\
\hline
6 & -1 & -1 & 1 & -3 & 0 & -5 & 0 & -8 & -4 & -12 & -3 & -16 & -9 & -25 & -11\\
\hline
8 & -2 & -2 & 0 & -1 & 0 & -7 & -5 & -11 & -5 & -16 & -7 & -24 & -18 & -34 & -20\\
\hline
10 & 0 & -2 & -1 & -4 & -1 & -7 & -5 & -15 & -8 & -19 & -10 & -32 & -24 & -42 & -29\\
\hline
12 & -2 & -3 & 0 & -7 & -5 & -11 & -9 & -20 & -17 & -30 & -20 & -40 & -34 & -59 & -44\\
\hline
14 & -3 & -4 & -1 & -7 & -5 & -15 & -12 & -25 & -20 & -34 & -26 & -52 & -47 & -72 & -57\\
\hline
\caption{Values of homological Euler Characteristics $\chi_h(\Sp(\Z), \m_\lambda)$}\label{homologicaleulermlambda}
\end{longtable}
\end{center}

\section{Cuspidal Cohomology} \label{Dim cuspidal}

In this section we give dimension formulas for the space of cuspidal cohomology. We will make use of Proposition 1 of \cite{MokTil2002}, which implies that, when we extend scalars to $\m_\lambda \otimes \C$ the cuspidal and inner cohomology coincides and is concentrated in degree $3$ (therefore the dimension of cuspidal cohomology is the same as the dimension, over $\Q$, of the inner cohomology). We denote by $h_!^q(\lambda)$ and $h_{Eis}^q(\lambda)$ the dimensions of the inner and the Eisenstein cohomology in degree $q$ of $\rS_\Gamma$ with coefficients in $\tm_\lambda$, respectively. Therefore one has
\[
dim(H^q(\rS_\Gamma, \tm_\lambda)) = h_!^q(\lambda) + h_{Eis}^q(\lambda).
\] 

We denote
\[
\chi_{Eis}(\Gamma, \m_\lambda) =  \sum_{k = 0}^5 (-1)^k h_{Eis}^k(\lambda).
\]

The following lemma will be useful in our analysis.

\begin{lema} \label{CuspDim}
One has the following dimension formula for the space of cuspidal cohomology:
\begin{align}
dim(H_{cusp}^\bullet(\rS_\Gamma, \tm_\lambda \otimes \C)) &= dim(H_{cusp}^3(\rS_\Gamma, \tm_\lambda \otimes \C)) = h_!^3(\lambda)\nonumber \\
&= \chi_{Eis}(\Gamma, \m_\lambda) - \chi_{h}(\Gamma, \m_\lambda)\nonumber 
\end{align}
\end{lema}

\begin{proof}
This follows from the aforementioned proposition in \cite{MokTil2002}.
\end{proof}

We will give an explicit description of $\chi_{Eis}(\Gamma, \m_\lambda)$ by using Section ~\ref{Eisenstein} and the fact that for an even integer $k \geq 4$,

$$\mbox{dim }\mathcal{S}_k  = 
\left\{ 
\begin{tabular}{rl}
$\lfloor \frac{k}{12} \rfloor - 1,$ & if $k \equiv 2 \mod 12$\\
$\lfloor \frac{k}{12} \rfloor,$ & otherwise\\
\end{tabular}\,.
\right.$$

\begin{prop} 
If $\lambda = m_1 \lambda_1 + m_2 \lambda_2$ then $\chi_{Eis}(\Gamma, \m_\lambda)  = 0$ for $m_1$ odd. If $m_1$ is even one has
$$\chi_{Eis}(\Gamma, \m_\lambda)  = 
\left\{ 
\begin{tabular}{rl}
$2,$ & if $m_1 = m_2 =0$ \\
dim $\mathcal{S}_{m_1 + 2} - 1 -$ dim $\mathcal{S}_{m_1+4},$ & if $m_1 > 0$ and $m_2 =0$\\
$2$dim $\mathcal{S}_{m_2 + 2} + 2 (\# \mathcal{Z}_{2m_2 + 4}) -$ dim $\mathcal{S}_{2m_2 + 4},$ & if $m_1 = 0$ and $m_2 > 0$ even\\
$1 + 2$dim $\mathcal{S}_{m_2 + 2} +$ dim $\mathcal{S}_{m_1 + 2} -$ dim $\mathcal{S}_{m_1+ 2m_2 + 4},$ & if $m_1 > 0$ and $m_2 > 0$ even\\
$- 1 - 2$dim $\mathcal{S}_{m_2 + 3} -$ dim $\mathcal{S}_{2m_2 + 4},$ & if $m_1 = 0$ and $m_2$ odd\\
dim $\mathcal{S}_{m_1 + 2} - 2$dim $\mathcal{S}_{m_1+ m_2 + 3} -$ dim $\mathcal{S}_{m_1+ 2m_2 + 4},$ & if $m_1 > 0$ and $m_2$ odd\\
\end{tabular}
\right. ,$$
where $\# \mathcal{Z}_{k}$ denotes the cardinality of $\mathcal{Z}_{k}$ (see Section ~\ref{Eisenstein}).
\end{prop}

It is therefore clear that $\chi_{Eis}(\Gamma, \m_\lambda)$ as a function of $\lfloor\frac{m_1}{12}\rfloor$ and $\lfloor\frac{m_2}{12}\rfloor$ only depends on the congruences of $m_1$ and $m_2$ modulo $12$. As a result we can give an explicit formula for $\chi_{Eis}(\Gamma, \m_\lambda)$. We introduce the matrix

\[
F_1=\left( 
\scalemath{0.7}{
\begin{array}{cccccccccccc} 
-2 & -1 & -1 & -1 & -2 & 0\\
-(2k_1+4k_2+1) & -(2k_1+4k_2) & -(2k_1+4k_2) & -(2k_1+4k_2+1) & -(2k_1+4k_2+2) & -(2k_1+4k_2)\\
0 & 1 & 0 & 1 & 0 & 1\\
-(2k_1+4k_2+1) & -(2k_1+4k_2+1) & -(2k_1+4k_2) & -(2k_1+4k_2+3) & -(2k_1+4k_2+1) & -(2k_1+4k_2+2)\\
-1 & 1 & 0 & 0 & 0 & 1\\
-(2k_1+4k_2+1) & -(2k_1+4k_2+1) & -(2k_1+4k_2+3) & -(2k_1+4k_2+1) & -(2k_1+4k_2+3) & -(2k_1+4k_2+3)\\
-1 & 0 & 0 & 0 & -1 & 1\\
-(2k_1+4k_2+2) & -(2k_1+4k_2+3) & -(2k_1+4k_2+1) & -(2k_1+4k_2+4) & -(2k_1+4k_2+3) & -(2k_1+4k_2+3)\\
-1 & 0 & -1 & 0 & -1 & 0\\
-(2k_1+4k_2+4) & -(2k_1+4k_2+2) & -(2k_1+4k_2+3) & -(2k_1+4k_2+4) & -(2k_1+4k_2+4) & -(2k_1+4k_2+3)\\
0 & 2 & 1 & 1 & 1 & 2\\
-(2k_1+4k_2+2) & -(2k_1+4k_2+4) & -(2k_1+4k_2+4) & -(2k_1+4k_2+4) & -(2k_1+4k_2+4) & -(2k_1+4k_2+6)\\
\end{array} 
}
\right),
\]
the vector 
\[ F_2 = \left( 
\scalemath{0.8}{
\begin{array}{cccccccccccc} 
-2 & -1 & -1 & -1 & -2 & 0
\end{array} 
}
\right),
\]
and the vector 
\[
F_3 = \left( 
\scalemath{0.8}{
\begin{array}{cccccccccccc} 
z-2, & -1-4k_2, & z, & -1-4k_2, & z-1, & -1-4k_2, & z-1, & -2-4k_2, & z-1, & -4-4k_2, & z, & -2-4k_2  \end{array} 
}
\right),
\]
where in $F_3$ the term $z$ denotes $2(\# \mathcal{Z}_{2m_2+4})$.
Then one can see the following:
\begin{prop} \label{EisensteinEulerCharacteristic}
If $m_1 = 12k_1+l_1$ and $m_2 = 12 k_2 + l_2$, with $k_1, k_2, l_1, l_2 \in \Z$ and $0 \leq l_1, l_2 < 12$ then
$$\chi_{Eis}(\Gamma, \m_\lambda) = 
\left\{ 
\begin{tabular}{rl}
$F_1(l_2, \frac{l_1}{2}),$ & $m_1 > 0 \mbox{ and } m_2 > 0$\\
$F_2(\frac{l_1}{2}),$ & $m_1 > 0 \mbox{ and } m_2 = 0$\\
$F_3(l_2),$ & $m_1 = 0 \mbox{ and } m_2 > 0$ \\
\end{tabular} \right. .$$
(where we are numbering the rows of $F_1$ and the entries of $F_3$ from $0$ to $11$ and the columns of $F_1$ and the entries of $F_2$ are numbered from $0$ to $5$).
\end{prop}

From Lemma~\ref{CuspDim} and Theorem ~\ref{Homological formula} we can deduce:

\begin{thm} \label{Cuspidal Dim}
If $\lambda = m_1 \lambda_1 + m_2\lambda_2$, with $m_1$ even, then the dimension $dim(H_{cusp}^\bullet(\rS_\Gamma, \tm_\lambda \otimes \C))$ of the cuspidal cohomology is given by
$$-\left(\tilde{\chi}_A(n_1,  n_2) + \tilde{\chi}_{B,C}(n_1,  n_2) + \tilde{\chi}_K(n_1,  n_2) + \tilde{\chi}_N(n_1,  n_2) + E(m_2, \frac{m_1}{2})\right) + 
\left\{ 
\begin{tabular}{rl}
$2,$ & $m_1 = m_2 = 0$\\
$F_1(l_2, \frac{l_1}{2}),$ & $m_1, m_2 > 0$\\
$F_2(\frac{l_1}{2}),$ & $m_1 > m_2 = 0$\\
$F_3(l_2),$ & $m_2 > m_1 = 0$ \\
\end{tabular} \right. .$$
\end{thm}

To illustrate the simplicity of this formula, we will give an example. Let $\lambda = 18 \lambda_1 + 10 \lambda_2$, then $m_1 = 18, m_2=10, n_1 = m_1 +m_2 = 28$ and $n_2 =m_2$. In what follows we determine each term,
\begin{itemize}
\item $m_2 = 10 = 12 \cdot 0 + 10$ and $m_1 = 18 = 1 \cdot 12 + 6$. Then $$\chi_{Eis}(\Gamma, \m_\lambda) = F_1(10, 3) = 1.$$
\item $m_2 = 10$ and $\frac{m_1}{2} = 9 \equiv 3$ mod $6$, therefore $$E(m_2, \frac{m_1}{2}) = \frac{43n_1}{432}+\frac{43}{216}$$
\item One has 
\begin{align*}
\tilde{\chi}_A(n_1,  n_2) &= -\frac{1}{720}\frac{1}{6}(n_1+2)(n_2+1)((n_1+2)^2-(n_2+1)^2)
\end{align*}
\item $n_1 \equiv n_2 \equiv 0$ mod $2$, therefore $$\tilde{\chi}_{B, C}(n_1,  n_2) = \frac{7}{144} \frac{(n_1+2)(n_2+1)}{2}$$
\item $n_1 \equiv 3$ and $n_2 \equiv 0$ mod $5$, therefore $$\tilde{\chi}_K(n_1,  n_2) = 0.$$
\item $n_1 \equiv 4$ mod $8$ and $n_2 \equiv 2$ mod $8$, therefore $$\tilde{\chi}_N(n_1,  n_2) = -\frac{1}{2}$$
\end{itemize}
And finally, for $\lambda = 18\lambda_1 + 10\lambda_2$, one has 
\begin{align*}
dim(H_{cusp}^\bullet(\rS_\Gamma, \tm_\lambda \otimes \C)) &= -\left(\frac{43 \cdot 28}{432}+\frac{43}{216}-\frac{1}{720}\frac{1}{6}(30 \cdot 11)(30^2-11^2)+\frac{7}{144} \frac{30 \cdot 11}{2} - \frac{1}{2}\right) + 1 \\
&= 50.
\end{align*}

We finally give a table with the dimensions of the cuspidal cohomology for the representations $\m_\lambda$, for $\lambda = m_1\lambda_1 + m_2 \lambda_2$ with $0 \leq m_1 < 15$ even and $0 \leq m_2 < 15$. One can use Theorem ~\ref{Cuspidal Dim} to easily obtain the dimension of the cuspidal cohomology for any other highest weight, as illustrated in the precedent example.

{
\begin{center}
\scriptsize\renewcommand{\arraystretch}{2}
\begin{longtable}{|c||c|c|c|c|c|c|c|c|c|c|c|c|c|c|c|} 
\hline
$m_1 \backslash m_2$ & 0& 1& 2& 3& 4& 5& 6& 7& 8& 9& 10& 11& 12& 13& 14 \\
\hline
\hline
0 & 0  & 0  & z  & 0  & z-2  & 0  & z-2  & 2  & z-2  & 2  & z-4  & 2  & z-4  & 4 & z-4 \\
\hline
2 & 0  & 0  & 0  & 0  & 0  & 0  & 0  & 0  & 0  & 0  & 0  & 4  & 0  & 8  & 0 \\
\hline
4 & 0  & 0  & 0  & 0  & 0  & 0  & 0  & 4  & 0  & 4  & 0  & 8  & 4  & 12  & 4 \\
\hline
6 & 0  & 0  & 0  & 0  & 0  & 4  & 0  & 4  & 4  & 8  & 4  & 12  & 8  & 20  & 12 \\
\hline
8 & 0  & 0  & 0  & 0  & 0  & 4  & 4  & 8  & 4  & 12  & 8  & 20  & 16  & 28  & 20 \\
\hline
10 & 0  & 2  & 2  & 2  & 2  & 4  & 6  & 12  & 8  & 16  & 12  & 26  & 24  & 38  & 30 \\
\hline
12 & 0  & 0  & 0  & 4  & 4  & 8  & 8  & 16  & 16  & 24  & 20  & 36  & 32  & 52  & 44 \\
\hline
14 & 2  & 2  & 2  & 4  & 6  & 12  & 12  & 20  & 20  & 30  & 28  & 46  & 46  & 66  & 58 \\
\hline
\hline
\caption{Dimension of cuspidal cohomology, for $\lambda = m_1 \lambda_1 + m_2 \lambda_2$}\label{cuspidal}
\end{longtable}
\end{center}
Where in the places $m_1 = 0$ and $m_2 > 0$ even,$``$z$"$ denotes $2(\# \mathcal{Z}_{2m_2 + 4})$.
}

\section{Dimension of $H^q(\Sp(\Z), \m_\lambda)$}\label{Dim Coho}
In this last section we describe the dimensions of the cohomology spaces for $\Gamma$ in every degree. These formulas will depend on the parity of the coefficients of the highest weight $\lambda = m_1 \lambda_1 + m_2 \lambda_2$. Let $h^q(\lambda)$ denote the dimension of $H^q(\Gamma, \m_\lambda)$. If $m_1$ is odd then $H^q(\Gamma, \m_\lambda) = 0$. On the other hand, if $m_1$ is even, one has $h^1(\lambda) = 0$ and:

\begin{eqnarray*}
h^0(\lambda) & = & \left\{ 
\begin{tabular}{rl}
$1,$ & if $m_1 = m_2 =0$ \\
$0,$ & otherwise
\end{tabular}
\right. ,
h^2(\lambda)  =  \left\{ 
\begin{tabular}{rl}
$1,$ & if $m_1 = m_2 =0$ \\
$\# \mathcal{Z}_{2m_2+4},$ & if $m_1 = 0$ and $m_2 \neq 0$ even\\
$0,$ & otherwise 
\end{tabular}
\right. ,\\
\\
h^3(\lambda)  & = & h_{cusp}^3(\lambda) + 
\left\{ 
\begin{tabular}{rl}
$0,$ & if $m_1 = m_2 =0$ \\
dim $\mathcal{S}_{2m_2 + 4} - \# \mathcal{Z}_{2m_2 + 4},$ & if $m_1 = 0$ and $m_2 > 0$ even\\
$1 +$ dim $\mathcal{S}_{m_1+4},$ & if $m_1 > 0$ and $m_2 =0$\\
dim $\mathcal{S}_{m_1+ 2m_2 + 4},$ & if $m_1 > 0$ and $m_2 > 0$ even\\
$1 + 2$dim $\mathcal{S}_{m_2 + 3} +$ dim $\mathcal{S}_{2m_2 + 4},$ & if $m_1 = 0$ and $m_2$ odd\\
$2$dim $\mathcal{S}_{m_1+ m_2 + 3} + $ dim $\mathcal{S}_{m_1+ 2m_2 + 4},$ & if $m_1 > 0$ and $m_2$ odd\\
\end{tabular}
\right. ,\\
\\
h^4(\lambda) & = &
\left\{ 
\begin{tabular}{rl}
$2$dim $\mathcal{S}_{m_2 + 2},$ & if $m_1 = 0$ and $m_2 > 0$ even\\
dim $\mathcal{S}_{m_1 + 2},$ & if $m_1 > 0$ and $m_2 =0$\\
$1 + 2$dim $\mathcal{S}_{m_2 + 2} +$ dim $\mathcal{S}_{m_1 + 2},$ & if $m_1 > 0$ and $m_2 > 0$ even\\
dim $\mathcal{S}_{m_1 + 2},$ & if $m_1 > 0$ and $m_2$ odd\\
$0,$ & otherwise
\end{tabular}
\right..
\end{eqnarray*}

Finally, let $h_{cusp}(\lambda), h_{Eis}(\lambda)$ and $h(\lambda)$ be the dimensions of the whole cuspidal, Eisenstein and the group cohomology of $\Gamma$ with respect to $\m_\lambda$ (i.e. the dimensions of $H_{cusp}^\bullet(\Sp(\Z), \tm_\lambda)$, $H_{Eis}^\bullet(\Sp(\Z), \tm_\lambda)$ and $H^\bullet(\Sp(\Z), \tm_\lambda)$), respectively. In particular one has $h_{cusp}(\lambda) = h_{cusp}^3(\lambda)$, $ h_{Eis}(\lambda) = \sum_{q=0}^4 h_{Eis}^q(\lambda)$ and $h(\lambda) = h_{cusp}(\lambda) + h_{Eis}(\lambda)$. Therefore: 
$$h(\lambda) = h_{cusp}(\lambda) + 
\left\{ 
\begin{tabular}{rl}
$2,$ & if $m_1 = m_2 =0$ \\
$2$dim $\mathcal{S}_{m_2 + 2}+$ dim $\mathcal{S}_{2m_2 + 4},$ & if $m_1 = 0$ and $m_2 > 0$ even\\
dim $\mathcal{S}_{m_1 + 2} + 1 + $ dim $\mathcal{S}_{m_1+4},$ & if $m_1 > 0$ and $m_2 =0$\\
$1 + 2$dim $\mathcal{S}_{m_2 + 2} +$ dim $\mathcal{S}_{m_1 + 2} +$ dim $\mathcal{S}_{m_1+ 2m_2 + 4},$ & if $m_1 > 0$ and $m_2 > 0$ even\\
$1 + 2$dim $\mathcal{S}_{m_2 + 3} +$ dim $\mathcal{S}_{2m_2 + 4},$ & if $m_1 = 0$ and $m_2$ odd\\
dim $\mathcal{S}_{m_1 + 2} + 2$dim $\mathcal{S}_{m_1+ m_2 + 3} +$ dim $\mathcal{S}_{m_1+ 2m_2 + 4},$ & if $m_1 > 0$ and $m_2$ odd\\
\end{tabular}
\right..$$

One has 
\begin{align}
h&(\lambda) = h_{Eis}(\lambda) + h_{cusp}(\lambda) \nonumber \\
&= h_{Eis}(\lambda) + \chi_{Eis}(\Gamma, \m_\lambda) - \left(\tilde{\chi}_A(n_1,  n_2) + \tilde{\chi}_{B,C}(n_1,  n_2) + \tilde{\chi}_K(n_1,  n_2) + \tilde{\chi}_N(n_1,  n_2) + E(m_2, \frac{m_1}{2})\right). \nonumber
\end{align}

The values of $\chi_{Eis}(\Gamma, \m_\lambda) + h_{Eis}(\lambda) = 2(h^0(\lambda) + h^2(\lambda) + h^4(\lambda))$ as a funtion of $\lfloor\frac{m_1}{12}\rfloor$ and $\lfloor\frac{m_2}{12}\rfloor$ will clearly depend on $m_1$ and $m_2$ modulo $12$. We will determine these values in the following matrices.  Consider

\[
G_1=\left( 
\scalemath{0.7}{
\begin{array}{cccccccccccc} 
2k_1 + 4k_2-4 & 2k_1 + 4k_2-2 & 2k_1 + 4k_2-2 & 2k_1 + 4k_2-2 & 2k_1 + 4k_2-2 & 2k_1 + 4k_2\\
2k_1-2 & 2k_1 & 2k_1 & 2k_1 & 2k_1 & 2k_1 + 2\\
2k_1 + 4k_2 & 2k_1 + 4k_2 + 2 & 2k_1 + 4k_2 + 2 & 2k_1 + 4k_2 + 2 & 2k_1 + 4k_2 + 2 & 2k_1 + 4k_2 + 4\\
2k_1-2 & 2k_1 & 2k_1 & 2k_1 & 2k_1 & 2k_1 + 2\\
2k_1 + 4k_2 & 2k_1 + 4k_2 + 2 & 2k_1 + 4k_2 + 2 & 2k_1 + 4k_2 + 2 & 2k_1 + 4k_2 + 2 & 2k_1 + 4k_2 + 4\\
2k_1-2 & 2k_1 & 2k_1 & 2k_1 & 2k_1 & 2k_1 + 2\\
2k_1 + 4k_2 & 2k_1 + 4k_2 + 2 & 2k_1 + 4k_2 + 2 & 2k_1 + 4k_2 + 2 & 2k_1 + 4k_2 + 2 & 2k_1 + 4k_2 + 4\\
2k_1-2 & 2k_1 & 2k_1 & 2k_1 & 2k_1 & 2k_1 + 2\\
2k_1 + 4k_2 & 2k_1 + 4k_2 + 2 & 2k_1 + 4k_2 + 2 & 2k_1 + 4k_2 + 2 & 2k_1 + 4k_2 + 2 & 2k_1 + 4k_2 + 4\\
2k_1-2 & 2k_1 & 2k_1 & 2k_1 & 2k_1 & 2k_1 + 2\\
2k_1 + 4k_2 + 4 & 2k_1 + 4k_2 + 6 & 2k_1 + 4k_2 + 6 & 2k_1 + 4k_2 + 6 & 2k_1 + 4k_2 + 6 & 2k_1 + 4k_2 + 8\\
2k_1-2 & 2k_1 & 2k_1 & 2k_1 & 2k_1 & 2k_1 + 2\\
\end{array} 
}
\right),
\]
the vector:
\[
G_2 = \left( 
\scalemath{0.8}{
\begin{array}{cccccccccccc} 
2k_1 -2, & 2k_1,  & 2k_1, & 2k_1, & 2k_1, & 2k_1 + 2
\end{array} 
}
\right),
\]
and the vector $G_3$ given by
\[
\left( 
\scalemath{0.8}{
\begin{array}{cccccccccccc} 
4k_2-4+z, & 0, & 4k_2+z, & 0, & 4k_2+z, & 0, & 4k_2+z, & 0, & 4k_2+z, & 0, & 4k_2 +4+z, & 0  \end{array} 
}
\right).
\]

\begin{prop}
If $m_1 = 12k_1+l_1$ and $m_2 = 12 k_2 + l_2$, with $k_1, k_2, l_1, l_2 \in \Z$ and $0 \leq l_1, l_2 < 12$ then
$$h_{Eis}(\lambda) + \chi_{Eis}(\Gamma, \m_\lambda) = 
\left\{ 
\begin{tabular}{rl}
$G_1(l_2, \frac{l_1}{2}),$ & $m_1 > 0 \mbox{ and } m_2 > 0$\\
$G_2(\frac{l_1}{2}),$ & $m_1 > 0 \mbox{ and } m_2 = 0$\\
$G_3(l_2),$ & $m_1 = 0 \mbox{ and } m_2 > 0$ \\
\end{tabular} \right. .$$
(where we are numbering the rows of $G_1$ and the entries of $G_3$ from $0$ to $11$ and the columns of $G_1$ and the entries of $G_2$ are numbered from $0$ to $5$).
\end{prop}

From Theorem~\ref{Homological formula} we can deduce:

\begin{thm} \label{Dim}
If $\lambda = m_1 \lambda_1 + m_2\lambda_2$, with $m_1$ even, then the dimension $dim(H^\bullet(\Sp(\Z), \m_\lambda))$ of the cohomology of $\Sp(\Z)$ is given by
$$-\left(\tilde{\chi}_A(n_1,  n_2) + \tilde{\chi}_{B,C}(n_1,  n_2) + \tilde{\chi}_K(n_1,  n_2) + \tilde{\chi}_N(n_1,  n_2) + E(m_2, \frac{m_1}{2})\right) + 
\left\{ 
\begin{tabular}{rl}
$2,$ & $m_1 = m_2 = 0$\\
$G_1(l_2, \frac{l_1}{2}),$ & $m_1, m_2 > 0$\\
$G_2(\frac{l_1}{2}),$ & $m_1 > m_2 = 0$\\
$G_3(l_2),$ & $m_2 > m_1 = 0$ \\
\end{tabular} \right. .$$
\end{thm}

We give an example of the use of this formula. Let $\lambda = 18 \lambda_1 + 70 \lambda_2$ then $m_1 = 18, m_2=70, n_1 = m_1 +m_2 = 88$ and $n_2 =m_2$.
\begin{itemize}
\item $m_2 = 70 = 12 \cdot 5 + 10$ and $m_1 = 18 = 1 \cdot 12 + 6$. Then $$\chi_{Eis}(\Gamma, \m_\lambda) = G_1(10, 3) = 2+20+6 = 28.$$
\item $m_2 = 70$ and $\frac{m_1}{2} = 9 \equiv 3$ mod $6$, therefore $E(m_2, \frac{m_1}{2}) = E(10, 3) = \frac{43n_1}{432}+\frac{43}{216}$.
\item $\tilde{\chi}_A(n_1,  n_2) = -\frac{1}{720}\frac{1}{6}(n_1+2)(n_2+1)((n_1+2)^2-(n_2+1)^2)$.
\item $n_1 \equiv n_2 \equiv 0$ mod $2$, therefore $\tilde{\chi}_{B, C}(n_1,  n_2) = \frac{7}{144} \frac{(n_1+2)(n_2+1)}{2}$.
\item Ver $n_1 \equiv 3$ and $n_2 \equiv 0$ mod $5$, therefore $\tilde{\chi}_K(n_1,  n_2) = 0$.
\item Ver $n_1 \equiv 0$ mod $8$ and $n_2 \equiv 6$ mod $8$, therefore $\tilde{\chi}_N(n_1,  n_2) = -\frac{1}{2}$.
\end{itemize}
And finally, for $\lambda = 18\lambda_1 + 70\lambda_2$, the dimension $dim(H^\bullet(\Sp(\Z), \m_\lambda))$ is
\[ 
-\left(\frac{43 \cdot 88}{432}+\frac{43}{216}-\frac{1}{720}\frac{1}{6}(90 \cdot 71)(90^2-71^2)+\frac{7}{144} \frac{90 \cdot 71}{2} - \frac{1}{2}\right) + 28 = 4389.
\]

In the following table we give the values of $h(\lambda)$ for $0 \leq m_1 < 15$ even and $0 \leq m_2 < 15$.

\begin{center}
\scriptsize\renewcommand{\arraystretch}{2}
\begin{longtable}{|c||c|c|c|c|c|c|c|c|c|c|c|c|c|c|c|c|c|c|c|c|c|c|c|c|c|c|c|c|c|c|c|} 
\hline
$m_1 \backslash m_2$ & 0& 1& 2& 3& 4& 5& 6& 7& 8& 9& 10& 11& 12& 13& 14 \\
\hline
\hline
0 & 2  & 1  & z  & 1  & z-1 & 1 & z-1 & 4  & z-1 & 6 & z & 4 & z-2 & 9 & z \\
\hline
2 & 1  & 0  & 1  & 1  & 1  & 1  & 2  & 3  & 2  & 2  & 4  & 8  & 3  & 12  & 5 \\
\hline
4 & 1  & 0  & 2  & 0  & 2  & 3  & 2  & 5  & 3  & 7  & 5  & 12  & 7  & 16  & 10 \\
\hline
6 & 1  & 1  & 1  & 3  & 2  & 5  & 2  & 8  & 6  & 12  & 9  & 16  & 11  & 25  & 17 \\
\hline
8 & 2  & 2  & 2  & 1  & 2  & 7  & 7  & 11  & 7  & 16  & 13  & 24  & 20  & 34  & 26 \\
\hline
10 & 2  & 4  & 5  & 6  & 5  & 9  & 9  & 17  & 12  & 21  & 18  & 34  & 28  & 44  & 37 \\
\hline
12 & 2  & 3  & 2  & 7  & 7  & 11  & 11  & 20  & 19  & 30  & 26  & 40  & 36  & 59  & 50 \\
\hline
14 & 5  & 6  & 5  & 9  & 9  & 17  & 16  & 27  & 24  & 36  & 34  & 54  & 51  & 74  & 65 \\
\hline
\caption{Dimension of $H^\bullet(\Sp(\Z), \m_\lambda)$, for $\lambda = m_1\lambda_1 + m_2\lambda_2$}\label{hbullet}
\end{longtable}
\end{center}

As in the previous section, when $m_1=0$ and $m_2>0$ is even, the ``z" denotes two times the cardinality of $ \mathcal{Z}_{2m_2+4}$.

\section*{Acknowledgements}

The authors would like to thank the Max Planck Institute for Mathematics (MPIM), Bonn, for its hospitality where the idea to pursue this project has been initiated during their visit in December, 2016. 

JB would like to thank the Mathematics Department of the Georg-August University G\"ottingen for the support, and especially to Valentin Blomer, Harald Helfgott and Thomas Schick for their immense support and encouragement during his postdoctoral studies. JB's work is financially supported by ERC Consolidator grant (Grant ID: 648329; codename GRANT). 

IH would like to thank Yves Martin for the long discussions on Siegel modular forms during his visit at the CUNY Graduate Center in Spring 2016.

MM would like to thank LAMA - Universit\'e Paris-Est Marne-la-Vall\'ee for the support and hospitality. MM also wants to thank Institut des Hautes Etudes Scientifiques and Institut Galil\'ee -Universit\'e Paris 13 for their hospitality, as part of this work took place during the stay of the third author in these institutions. MM is also thankful to Michael Harris, Nicolas Andruskiewitsch and Roberto Miatello for their support and encouragement.

Last but not the least, authors would like to extend their thanks to G\"unter Harder for many inspiring discussions on the subject and his support during the writing on this article.
\nocite{}
\bibliographystyle{abbrv}
\bibliography{BHM}

\end{document}